\documentclass[11pt]{article}

\usepackage[T1]{fontenc}
\usepackage[a4paper,margin=2.55cm]{geometry}
\usepackage{amsmath,amssymb,amsthm,mathtools,mathrsfs}
\usepackage{microtype}
\usepackage{booktabs}
\usepackage[hidelinks]{hyperref}
\usepackage{authblk}

\DeclareMathOperator{\End}{End}
\newtheorem{theorem}{Theorem}[section]
\newtheorem{proposition}[theorem]{Proposition}
\newtheorem{lemma}[theorem]{Lemma}
\newtheorem{corollary}[theorem]{Corollary}

\theoremstyle{definition}
\newtheorem{definition}[theorem]{Definition}

\theoremstyle{remark}
\newtheorem{remark}[theorem]{Remark}

\theoremstyle{plain}

\newcommand{\Wres}{\operatorname{Wres}}
\newcommand{\Wresb}{\widetilde{\operatorname{Wres}}}
\newcommand{\trE}{\operatorname{tr}_{E}}
\newcommand{\Id}{\operatorname{Id}}
\newcommand{\Ric}{\operatorname{Ric}}
\newcommand{\vol}{\operatorname{vol}}
\newcommand{\dd}{\,\mathrm{d}}
\newcommand{\ii}{\mathrm{i}}
\newcommand{\eps}{\varepsilon}
\newcommand{\iotaop}{\iota}
\newcommand{\cG}[1]{c_{g_{#1}}}

\title{The Bimetric generalization of Kastler--Kalau--Walze Type Theorems}

\author[1]{Sining Wei\thanks{ E-mail: \texttt{weisn835@nenu.edu.cn}}}
\author[2]{Yong Wang\thanks{Corresponding author. E-mail: \texttt{wangy581@nenu.edu.cn}}}
\affil[1]{School of Data Science and Artificial Intelligence, Dongbei University of Finance and Economics\\ Dalian, 116025, China}
\affil[2]{School of Mathematics and Statistics, Northeast Normal University, Changchun, 130024, China}

\date{}

\begin{document}
\maketitle
\vspace{-1.2cm}

\begin{abstract}
Let $M^n$ be a closed oriented manifold of even dimension
$n=2m\ge2$, equipped with a smooth metric $g_1$
and a Riemannian metric $g_2$.
We compute the Wodzicki residue of $D_{g_1}^2D_{g_2}^{-n}$ on
the common exterior bundle, where $D_{g_r}=d+\delta_{g_r}$. This is the noncommutative integral of $D_{g_1}^2$ relative
to the reference operator $D_{g_2}$.
Its local density involves the curvatures of the two metrics
and the difference of their Levi--Civita connections.
Integration by parts gives a closed-manifold formula without
explicit derivatives of the connection difference. For compact manifolds with boundary and $n\ge4$, we assume that both metrics are Riemannian and satisfy
$g_r=h_r(x_n)^{-1}g^{\partial M}+dx_n^2$
near the boundary, with $h_r(0)=1$.
We compute the noncommutative residues of two
factorizations involving even and odd powers of $D_{g_2}$.
Their interior contributions coincide, whereas their boundary
terms are explicit multiples of the integral of
$K_{g_1}-K_{g_2}$, with the coefficient for the even
factorization twice that for the odd one.
Here $K_{g_r}$ is the trace of the second fundamental form
with respect to the inward unit normal.
The calculation uses a second-order residue formula and
direct boundary symbol expansions.
When the metrics coincide, the formulas reduce to the
Hodge--de Rham Kastler--Kalau--Walze identity.
\end{abstract}



\section{Introduction}

The noncommutative residue relates the symbolic calculus of
elliptic operators to Riemannian curvature.
In the Kastler--Kalau--Walze theorem, the residue of a
negative $(n-2)$-power of a Dirac operator is proportional to the
Einstein--Hilbert action
\cite{Connes1994,KalauWalze1995,Kastler1995,Wodzicki1987}.
The relation with heat coefficients provides another approach
to the calculation of such residues
\cite{Ackermann1996,Gilkey1995}.
For manifolds with boundary, Boutet de Monvel's calculus
\cite{Boutet1971,Grubb1996} provides the corresponding algebra
of boundary operators. Its noncommutative residue contains
a boundary contribution in addition to the interior residue
density \cite{FGLS1996,GrubbSchrohe2001,Schrohe1999}.

In the usual construction, the same metric determines both
the differential operator and the elliptic negative power.
Separating these roles leads to a comparison problem for two
Riemannian metrics on a fixed manifold.
The principal symbols involve the two quadratic forms,
whereas the lower-order symbols also depend on their
Levi--Civita connections.
The resulting residue therefore raises two questions:
which curvature and connection terms occur in the
two-metric formula, and how do they combine to recover the
classical identity when the metrics agree?
A comparison of the scalar principal symbols alone does
not determine these lower-order contributions.

The Hodge--de Rham operator provides a natural setting for
this problem.
For every Riemannian metric $g$, the operator
$D_g=d+\delta_g$ acts on
$E=\Lambda^*T^*M\otimes\mathbb C$, where $\delta_g$ is
the formal adjoint of the exterior differential $d$.
The codifferential and the induced connection depend on
$g$, but the bundle and the exterior differential do not.
Thus two such operators can be composed without choosing
an identification between metric-dependent spinor bundles.

Several earlier results are particularly relevant to the
two-metric problem considered here.
Liu and Wang \cite{LiuWang2026} introduced a bimetric
spectral Einstein--Hilbert action for scalar Laplace
operators and computed it for Lorentz warped products.
Their construction separates the metric of the
second-order insertion from the Riemannian metric
determining the elliptic denominator.
It motivates the corresponding comparison for Hodge
Laplacians, whose lower-order symbols also involve
the induced connections on differential forms.
For a single Riemannian metric,
D\k{a}browski, Sitarz, and Zalecki
\cite{DabrowskiSitarzZalecki2025} computed the spectral
metric and Einstein functionals of the Hodge--Dirac
operator and identified them, up to dimensional
constants, with those of the canonical spin Dirac
operator.
Their normal-coordinate symbol expansions and
exterior-algebra trace identities supply the Hodge
coefficients needed in our interior calculation.
The second-order residue formula recalled by
Bochniak, D\k{a}browski, Sitarz, and Zalecki
\cite[Proposition~2.1]{BDSZ2026} expresses the residue
of a second-order differential insertion against a
negative power of a Laplace-type operator in terms
of its local coefficients.
Applying this formula allows us to keep the
two-metric numerator separate from the Hodge
denominator throughout the calculation.
Related residue functionals have also been studied
for perturbed Hodge--de Rham operators
\cite{WangWangLiu2025} and for a one-form rescaling
of the spin Dirac operator \cite{WangWang2025}.
For manifolds with boundary, Wang developed explicit
residue calculations in non-product collars
\cite{Wang2006Forms,Wang2006Nonproduct} and established
KKW-type formulas for the Dirac and signature
operators \cite{Wang2007}.
Liu, Wu, and Wang \cite{LiuWuWang2022} subsequently
derived Lichnerowicz-type formulas for perturbed
Hodge--de Rham operators and carried out boundary
residue calculations in dimensions four and six.
These works provide the collar construction and
the boundary symbol method used below.
Our boundary problem retains two independent
warping functions and compares two factorizations
of the same interior operator, so that the
dependence on the two metrics and on the
factorization can be examined separately.

Let $n=2m\ge2$. On a closed oriented manifold $M^n$,
equipped with Riemannian metrics $g_1$ and $g_2$, we
consider
\[
\mathcal K_n(g_1\mid g_2)
=
\Wres(D_{g_1}^2D_{g_2}^{-n}).
\]
The first metric determines the differential insertion,
and the second determines the elliptic denominator.
The metric arguments are therefore ordered.
Writing $C=\nabla^{g_1}-\nabla^{g_2}$ for the difference
of the Levi--Civita connections,
Theorem~\ref{thm:local-closed} expresses the residue
density in terms of the scalar curvatures, the Ricci
tensor of $g_2$, the connection difference $C$, and
the covariant derivative of its contraction.
On a closed manifold,
Theorem~\ref{thm:global-closed} removes the derivative
term by integration by parts.
The diagonal case $g_1=g_2$ gives the usual
Hodge--de Rham KKW formula.

We next consider the effect of factorization at the boundary.
Let $M^n$ be compact and oriented, with $n=2m\ge4$.
We assume that the two metrics induce the same
Riemannian metric $g^{\partial M}$ on $\partial M$ and have
the following form in a common collar:
\[
g_r=h_r(x_n)^{-1}g^{\partial M}+dx_n^2,
\qquad
h_r(0)=1,
\qquad r=1,2.
\]
Here each $h_r$ is a positive smooth function of the
inward normal coordinate $x_n$.
This is the two-metric version of the collar used in
\cite[(1.3)]{Wang2007}.
Denoting truncation to $M$ by $\pi^+$ and the
noncommutative residue in the boundary calculus by
$\Wresb$, we study
\begin{align*}
\mathcal R_n^{(2)}(g_1\mid g_2)
&=
\Wresb\!\left[
\pi^+(D_{g_1}^2D_{g_2}^{-2})
\circ
\pi^+(D_{g_2}^{-n+2})
\right],\\
\mathcal R_n^{(1)}(g_1\mid g_2)
&=
\Wresb\!\left[
\pi^+(D_{g_1}^2D_{g_2}^{-1})
\circ
\pi^+(D_{g_2}^{-n+1})
\right].
\end{align*}
We refer to these as the even and odd splittings,
respectively.
Their interior products agree modulo smoothing
operators, but their boundary residues need not
agree because truncation is not multiplicative.
Theorem~\ref{thm:boundary-KKW} gives the same interior
contribution for both splittings and expresses their
singular Green contributions as explicit dimensional
multiples of
$\int_{\partial M}(K_{g_1}-K_{g_2})\dd\vol_{g^{\partial M}}$.
The coefficient for the even splitting is twice
that for the odd splitting.
Here $K_{g_r}$ denotes the trace of the second
fundamental form with respect to the inward unit
normal.
In particular, the boundary contributions vanish
when the two mean curvatures agree.
These conclusions concern the warped collars
specified above.

For the interior calculation, we apply
\cite[Proposition~2.1]{BDSZ2026} with
$O=D_{g_1}^2$ and $L=D_{g_2}^2$.
Both operators are written in the coordinate exterior
frame associated with $g_2$-normal coordinates.
The denominator coefficients are those of the
Hodge Laplacian, whereas the numerator coefficients
are expressed through the connection difference $C$.
For the boundary calculation, we use the local
noncommutative residue formula in
\cite[(2.2.5)]{Wang2007}.
Its order condition leaves five cases for each
factorization.
We retain $h_1'(0)$ and $h_2'(0)$ separately in the
symbol expansions; their difference appears after
the cancellations.
Only one case contributes for the even splitting
and two for the odd splitting.

This paper is organized as follows. Section~2 gives the symbol conventions and exterior
trace identities.
Section~3 derives the interior formula.
Section~4 describes the warped collars and the
boundary symbols, and Section~5 evaluates both
factorizations and proves the boundary theorems.
The appendix records mixed Clifford trace identities
on the common exterior bundle.

\section{Preliminaries}
\label{sec:conventions}

In this section, we introduce the exterior-algebra, curvature, symbol, and residue conventions used below. Both Hodge--de Rham operators are
realized on the same exterior bundle. We also give the fiber-trace
identities needed in the interior and boundary calculations.

\subsection{The common exterior bundle and symbol conventions}

Let $M$ be a smooth oriented manifold of dimension $n$, and let
$g_1$ and $g_2$ be Riemannian metrics on $M$. Set
$E=\Lambda^*T^*M\otimes\mathbb C
=\bigoplus_{p=0}^n\Lambda^pT^*M\otimes\mathbb C$.
The definition of $E$ is independent of either metric, and
$\dim_{\mathbb C}E_x=2^n$ for every $x\in M$.

For local coordinates $(x^1,\ldots,x^n)$, write
$\partial_i=\partial/\partial x^i$ and
$g_{r,ij}=g_r(\partial_i,\partial_j)$, where $r\in\{1,2\}$,
and let $(g_r^{ij})$ be the inverse matrix of $(g_{r,ij})$.
Unless otherwise stated, indices run from $1$ to $n$, and repeated
indices are summed. We write $\dd\vol_{g_r}$ for the Riemannian
volume form and $\Id_E$ for the identity endomorphism of $E$.

Let $\alpha,\beta\in\Omega^1(M)$ and
$X,Y\in\Gamma(TM)$. Exterior multiplication and contraction are
defined by
\[
\eps(\alpha)\omega=\alpha\wedge\omega,
\qquad
\bigl(\iotaop(X)\omega\bigr)(X_1,\ldots,X_{p-1})
=\omega(X,X_1,\ldots,X_{p-1})
\]
for $\omega\in\Omega^p(M)$ with $p\geq1$; contraction is zero on
zero-forms. These operators satisfy the canonical anticommutation
relations
\begin{align}
\eps(\alpha)\eps(\beta)+\eps(\beta)\eps(\alpha)
&=0,\notag\\
\iotaop(X)\iotaop(Y)+\iotaop(Y)\iotaop(X)
&=0,\notag\\
\eps(\alpha)\iotaop(X)+\iotaop(X)\eps(\alpha)
&=\alpha(X)\Id_E.
\label{eq:CAR}
\end{align}

For a Riemannian metric $g$, let $\nabla^g$ denote its Levi--Civita
connection, with Christoffel coefficients determined by
$\nabla^g_{\partial_i}\partial_j
=\Gamma(g)^k{}_{ij}\partial_k$.
We use the curvature convention
\[
R^g(X,Y)Z=\nabla^g_X\nabla^g_YZ
-\nabla^g_Y\nabla^g_XZ
-\nabla^g_{[X,Y]}Z
\]
and write
$R^g_{abcd}
=g(R^g(\partial_a,\partial_b)\partial_d,\partial_c)$.
Accordingly,
$\Ric(g)_{bd}=g^{ac}R^g_{abcd}$ and
$R_g=g^{bd}\Ric(g)_{bd}$.

Choose $g$-normal coordinates
$x=(x^1,\ldots,x^n)$ centered at $x_0\in M$. Thus $x(x_0)=0$,
$g_{ab}(0)=\delta_{ab}$,
$\partial_cg_{ab}(0)=0$, and
$\Gamma(g)^k{}_{ij}(0)=0$.
Here $\delta_{ab}$ is the Euclidean metric at the origin. For
$\xi=\xi_a\,dx^a$, put
$|x|^2=\delta_{ab}x^ax^b$ and
$|\xi|^2=\delta^{ab}\xi_a\xi_b$.
With the curvature convention above, the standard normal-coordinate
expansions \cite[p.~1092]{DabrowskiSitarzZalecki2025} are
\begin{align*}
g_{ab}(x)
&=
\delta_{ab}
-\frac13R^g_{acbd}(x_0)x^cx^d
+O(|x|^3),\\
g^{ab}(x)
&=
\delta^{ab}
+\frac13R^g_{acbd}(x_0)x^cx^d
+O(|x|^3),\\
\sqrt{\det(g_{ab}(x))}
&=
1-\frac16\Ric(g)_{ij}(x_0)x^ix^j
+O(|x|^3).
\end{align*}
In particular,
$g^{ab}(x)\xi_a\xi_b
=(\delta^{ab}+\frac13R^g_{acbd}(x_0)x^cx^d)
\xi_a\xi_b+O(|x|^3|\xi|^2)$.

The contracted Christoffel coefficients satisfy
\[
\Gamma(g)^a{}_{ja}
=\frac12g^{ab}\partial_jg_{ab}
=\partial_j\log\sqrt{\det(g_{ab})}.
\]
The volume-density expansion gives
$\log\sqrt{\det(g_{ab}(x))}
=-\frac16\Ric(g)_{k\ell}(x_0)x^kx^\ell+O(|x|^3)$.
Differentiating twice at the origin and using the symmetry of the
Ricci tensor yields
\begin{equation}\label{eq:contracted-Gamma}
\partial_i\Gamma(g)^a{}_{ja}(0)
=-\frac13\Ric(g)_{ij}(x_0).
\end{equation}

For a vector field $X$, let $X_g^\flat=g(X,\cdot)$ and define
$c_g(X)=\eps(X_g^\flat)-\iotaop(X)$.
Equivalently, for a one-form $\alpha$, set
$c_g(\alpha)=\eps(\alpha)-\iotaop(\alpha^{\sharp_g})$, where
$\alpha^{\sharp_g}$ is determined by
$g(\alpha^{\sharp_g},X)=\alpha(X)$.
Thus, for $r\in\{1,2\}$,
$c_{g_r}(X)=\eps(X_{g_r}^\flat)-\iotaop(X)$ and
$c_{g_r}(\alpha)
=\eps(\alpha)-\iotaop(\alpha^{\sharp_{g_r}})$.
It follows from \eqref{eq:CAR} that
\[c_g(X)c_g(Y)+c_g(Y)c_g(X)
=-2g(X,Y)\Id_E
\]
and
\[c_g(\alpha)c_g(\beta)+c_g(\beta)c_g(\alpha)
=-2g^{-1}(\alpha,\beta)\Id_E.
\]
Let $\delta_{g_r}$ be the formal adjoint of the exterior differential
$d$ with respect to the $L^2$-inner product induced by $g_r$. The
Hodge--de Rham operator
$D_{g_r}=d+\delta_{g_r}$ acts on $C^\infty(M,E)$.
With the convention $D_{x^j}=-\ii\partial_{x^j}$, its principal
symbol satisfies
\[
\sigma_1(D_{g_r})(x,\xi)=\ii\,c_{g_r}(\xi),
\qquad
c_{g_r}(\xi)^2=-|\xi|_{g_r}^2\Id_E,
\qquad
|\xi|_{g_r}^2=g_r^{ij}\xi_i\xi_j.
\]
Consequently, $D_{g_r}^2$ is a Laplace-type operator with scalar
principal symbol
\[
\sigma_2(D_{g_r}^2)(x,\xi)
=|\xi|_{g_r}^2\Id_E.
\]
On a closed manifold, the negative powers of $D_{g_r}$ are defined
by spectral calculus on the orthogonal complement of
$\ker D_{g_r}$ and are set equal to zero on the kernel. Equivalently,
they may be replaced by parametrices modulo smoothing operators.
The two definitions differ by a smoothing operator and hence give
the same noncommutative residue. We denote the ideal of smoothing
operators by $\Psi^{-\infty}$.

If $P$ is a classical pseudodifferential operator of order $q$, write
$$\sigma(P)(x,\xi)\sim\sum_{j=0}^{\infty}
\sigma_{q-j}(P)(x,\xi),$$ where
$\sigma_{q-j}(P)$ is positively homogeneous of degree $q-j$ in
$\xi$ for $|\xi|\geq1$.
For a multi-index
$\alpha=(\alpha_1,\ldots,\alpha_n)\in\mathbb N_0^n$, set
$|\alpha|=\sum_{j=1}^n\alpha_j$,
$\alpha!=\prod_{j=1}^n\alpha_j!$,
$\partial_\xi^\alpha
=\prod_{j=1}^n\partial_{\xi_j}^{\alpha_j}$, and
$\partial_x^\alpha
=\prod_{j=1}^n\partial_{x^j}^{\alpha_j}$.
We use the left composition convention
\begin{equation}
\label{eq:symbol-product}
\sigma(PQ)(x,\xi)
\sim
\sum_{\alpha\in\mathbb N_0^n}
\frac{1}{\alpha!}
\partial_\xi^\alpha\sigma(P)(x,\xi)
D_x^\alpha\sigma(Q)(x,\xi),
\qquad
D_x^\alpha=(-\ii)^{|\alpha|}\partial_x^\alpha.
\end{equation}

Throughout the paper, $\trE$ denotes the ordinary complex fiber trace
on $\End(E_x)$, so that $\trE(\Id_E)=2^n$. We write
$\nu_j=\vol(S^j)$ for the Euclidean volume of the unit sphere $S^j$.
If $M$ is closed and $P$ is a classical pseudodifferential operator
on $E$, our residue convention is
\[
\Wres(P)
=
\int_M
\int_{|\xi|_{g_2}=1}
\trE\bigl(\sigma_{-n}(P)(x,\xi)\bigr)
\dd S_{g_2}(\xi)\,
\dd\vol_{g_2}(x),
\]
where $\sigma_{-n}(P)$ is the homogeneous component of degree $-n$
of the complete symbol of $P$, and $\dd S_{g_2}$ is the hypersurface
measure on the $g_2$-unit cotangent sphere.

\subsection{Exterior-module trace identities}
\label{sec:exterior-traces}

In the coordinate exterior basis, $\eps(dx^i)$ denotes exterior
multiplication by $dx^i$, while $\iotaop(\partial_i)$ denotes
contraction by the coordinate vector field $\partial_i$.

\begin{lemma}[Exterior trace identities]
\label{lem:single-metric-traces}
For $x\in M$, $X,Y\in T_xM$, and a Riemannian metric $g$,
\[
\trE\bigl(c_g(X)c_g(Y)\bigr)=-2^n g(X,Y).
\]
Moreover,
\begin{align}
\trE\bigl(\eps(dx^a)\iotaop(\partial_b)\bigr)
&=2^{n-1}\delta^a{}_b,
\label{eq:lambda2}\\
\trE\bigl(
\eps(dx^a)\iotaop(\partial_b)
\eps(dx^c)\iotaop(\partial_d)
\bigr)
&=
2^{n-2}
\left(
\delta^a{}_b\delta^c{}_d
+\delta^a{}_d\delta^c{}_b
\right).
\label{eq:lambda4}
\end{align}
\end{lemma}

\begin{proof}
The Clifford relation and cyclicity of the fiber trace give
\[
2\trE\bigl(c_g(X)c_g(Y)\bigr)
=
-2g(X,Y)\trE(\Id_E)
=
-2^{n+1}g(X,Y),
\]
which proves the first identity.

Applying \eqref{eq:CAR} with
$\alpha=dx^a$ and $X=\partial_b$, and then using cyclicity, gives
$2\trE(\eps(dx^a)\iotaop(\partial_b))
=\delta^a{}_b\trE(\Id_E)$.
This proves \eqref{eq:lambda2}.

For the four-factor trace, the anticommutation relations give
\begin{align*}
&\trE\bigl(
\eps(dx^a)\iotaop(\partial_b)
\eps(dx^c)\iotaop(\partial_d)
\bigr)
=
\delta^c{}_b\,
\trE\bigl(\eps(dx^a)\iotaop(\partial_d)\bigr)
-
\trE\bigl(
\eps(dx^a)\eps(dx^c)
\iotaop(\partial_b)\iotaop(\partial_d)
\bigr),
\end{align*}
while cyclicity and \eqref{eq:CAR} yield
\begin{align*}
&2\trE\bigl(
\eps(dx^a)\eps(dx^c)
\iotaop(\partial_b)\iotaop(\partial_d)
\bigr)
=
\delta^a{}_d\,
\trE\bigl(\eps(dx^c)\iotaop(\partial_b)\bigr)
-
\delta^c{}_d\,
\trE\bigl(\eps(dx^a)\iotaop(\partial_b)\bigr).
\end{align*}
Substitution of \eqref{eq:lambda2} gives \eqref{eq:lambda4}.
\end{proof}

\section{The interior residue}
\label{sec:universal}

We compute the interior residue of $D_{g_1}^2D_{g_2}^{-n}$
by applying a second-order residue formula with denominator
$D_{g_2}^2$. We first evaluate the relevant Hodge coefficients
and then express the numerator in the same coordinate exterior
frame. These calculations give the local bimetric density and
its integral on a closed manifold.

\subsection{The second-order residue formula and the Hodge denominator}
\label{subsec:second-order-residue}

Let $n=2m\geq2$. Choose $g_2$-normal coordinates centered at
$x_0\in M$ and a smooth local frame of $E$. We use the symbol
composition convention \eqref{eq:symbol-product}.

Let $L$ be a Laplace-type operator on $E$ with principal symbol
$g_2^{ab}(x)\xi_a\xi_b\Id_E$. Write
$$\sigma(L)=a_2+a_1+a_0,$$ where $a_j$ is homogeneous of degree
$j$ in $\xi$. Following the coefficient notation of
\cite[(2.1)]{BDSZ2026}, write
\begin{align}
a_2(x,\xi)
&=
g_2^{ab}(x)\xi_a\xi_b\Id_E,
\notag\\
a_1(x,\xi)
&=
\ii\bigl(P_{ab}x^b+S_a\bigr)\xi_a
+o(|x||\xi|),
\notag\\
a_0(x)
&=
Q+o(1).
\label{eq:L-symbol}
\end{align}
The normal-coordinate expansion of $g_2^{ab}$ was recalled
in Section~\ref{sec:conventions}. The remainders in
\eqref{eq:L-symbol} are understood as $x\to0$, uniformly for
$|\xi|=1$. The coefficients
$S_a,P_{ab},Q\in\End(E_{x_0})$ are characterized by
$a_1(0,\xi)=\ii S_a\xi_a$,
$\partial_{x^b}a_1(0,\xi)=\ii P_{ab}\xi_a$, and
$Q=a_0(0)$. They are computed in the chosen local frame;
no symmetry of $P_{ab}$ is assumed.

Let $O$ be a differential operator of order at most two.
At $x_0$, write
\[
\sigma(O)(x_0,\xi)
=
F^{ab}\xi_a\xi_b+\ii G^a\xi_a+H,
\qquad
F^{ab}=F^{ba},
\]
where $F^{ab},G^a,H\in\End(E_{x_0})$. The symmetry of
$F^{ab}$ may be assumed because its antisymmetric part
does not contribute to $F^{ab}\xi_a\xi_b$.
In the following formula, repeated coordinate indices are
contracted with the Euclidean metric at $x_0$, and products
of endomorphisms denote composition in the order written.

We recall the second-order residue formula of
\cite[Proposition~2.1]{BDSZ2026}.

\begin{proposition}[Second-order residue formula]
\label{prop:universal}
Suppose that $M$ is closed and that $L$ and $O$ have the
local symbols described above. Then
\begin{equation}
\label{eq:universal}
\begin{aligned}
\Wres(OL^{-m})
={}&
\frac{\nu_{n-1}}{24}
\int_M
\trE\Bigl[
24H+12G^aS_a
\\
&\quad
+F^{aa}
\bigl(
-12Q+6P_{bb}-2R_{g_2}\Id_E-3S_bS_b
\bigr)
\\
&\quad
+2F^{ab}
\bigl(
-6P_{ab}+2\Ric(g_2)_{ab}\Id_E-3S_aS_b
\bigr)
\Bigr]\dd\vol_{g_2}.
\end{aligned}
\end{equation}
The coefficients in the integrand are evaluated pointwise
in normal coordinates and the chosen local frame.
The negative power $L^{-m}$ is understood modulo smoothing
operators.
\end{proposition}

We also use \eqref{eq:universal} as a local identity.
Let $\chi$ be a smooth scalar function supported in a
coordinate neighborhood. Left multiplication gives
$\sigma(\chi O)=\chi\,\sigma(O)$ and
\[
\sigma_{-n}\bigl((\chi O)L^{-m}\bigr)
=
\chi\,\sigma_{-n}(OL^{-m}).
\]
Applying \eqref{eq:universal} to $\chi O$ therefore multiplies
both integrands by $\chi$. Since $\chi$ is arbitrary, the
traced sphere integral of $\sigma_{-n}(OL^{-m})$ equals
pointwise $\nu_{n-1}/24$ times the trace of the bracket in
\eqref{eq:universal}. By locality of the symbol calculation,
the same identity gives the interior density on a manifold
with boundary.

We now take $L=D_{g_2}^2$ and use the coordinate exterior
frame. The following identities follow from the Hodge symbol
calculation in
\cite[Lemma~2.1 and the proof of Proposition~2.3]
{DabrowskiSitarzZalecki2025}, together with the exterior
trace identities \eqref{eq:lambda2} and \eqref{eq:lambda4}.

\begin{lemma}[Traced Hodge coefficients]
\label{lem:Hodge-coeff}
For $L=D_{g_2}^2$ in the coordinate exterior frame of
$g_2$-normal coordinates centered at $x_0$, the coefficients
in \eqref{eq:L-symbol} satisfy $S_a=0$ and
\[
\trE P_{ab}
=
\frac{2^n}{3}\Ric(g_2)_{ab},
\qquad
\trE P_{aa}
=
\frac{2^n}{3}R_{g_2},
\qquad
\trE Q
=
\frac{2^{n-2}}{3}R_{g_2}.
\]
All quantities are evaluated at $x_0$.
\end{lemma}

\begin{proof}
By \cite[Lemma~2.1]{DabrowskiSitarzZalecki2025}, the
first-order Hodge symbol in the coordinate exterior frame
has the expansion
\[
a_1(x,\xi)
=
\ii\left[
\frac23\Ric(g_2)_{ab}\Id_E
-\frac23\eps(dx^k)\iotaop(\partial_\ell)
\bigl(
R^{g_2}_{\ell kab}+R^{g_2}_{\ell akb}
\bigr)
\right]x^b\xi_a
+o(|x||\xi|),
\]
where the curvature components are evaluated at $x_0$.
In particular, $a_1(0,\xi)=0$ for every $\xi$.
On the other hand, \eqref{eq:L-symbol} gives
$a_1(0,\xi)=\ii S_a\xi_a$. Since the components $\xi_a$
can be chosen independently, $S_a=0$ for every $a$.

Comparing the coefficients of $\ii x^b\xi_a$ determines
$P_{ab}$. The zero-order symbol calculation in
\cite[proof of Proposition~2.3]{DabrowskiSitarzZalecki2025}
gives $Q$. Thus
\begin{align*}
P_{ab}
&=
\frac23\Ric(g_2)_{ab}\Id_E
-\frac23\eps(dx^k)\iotaop(\partial_\ell)
\bigl(
R^{g_2}_{\ell kab}+R^{g_2}_{\ell akb}
\bigr),
\\
Q
&=
-\frac13
\eps(dx^p)\iotaop(\partial_q)
\eps(dx^k)\iotaop(\partial_\ell)
\bigl(
R^{g_2}_{\ell qkp}+R^{g_2}_{\ell kqp}
\bigr).
\end{align*}

Taking the trace of $P_{ab}$ and using
\eqref{eq:lambda2}, we obtain
\begin{align*}
\trE P_{ab}
&=
\frac{2^{n+1}}3\Ric(g_2)_{ab}
-\frac{2^n}{3}
\sum_k
\bigl(
R^{g_2}_{kkab}+R^{g_2}_{kakb}
\bigr)
\\
&=
\frac{2^{n+1}}3\Ric(g_2)_{ab}
-\frac{2^n}{3}\Ric(g_2)_{ab}
\\
&=
\frac{2^n}{3}\Ric(g_2)_{ab}.
\end{align*}
Here $R^{g_2}_{kkab}=0$ by antisymmetry in the first two
curvature indices, and
$\sum_kR^{g_2}_{kakb}=\Ric(g_2)_{ab}$ by the definition
of the Ricci tensor. Contracting with $\delta^{ab}$ gives
$\trE P_{aa}=2^nR_{g_2}/3$.

For the zero-order coefficient, \eqref{eq:lambda4} yields
\begin{align*}
\trE Q
&=
-\frac{2^{n-2}}3
\bigl(
\delta^p{}_q\delta^k{}_\ell
+\delta^p{}_\ell\delta^k{}_q
\bigr)
\bigl(
R^{g_2}_{\ell qkp}+R^{g_2}_{\ell kqp}
\bigr)
\\
&=
-\frac{2^{n-2}}3
\sum_{p,k}
\bigl(
R^{g_2}_{kpkp}
+R^{g_2}_{kkpp}
+2R^{g_2}_{pkkp}
\bigr)
\\
&=
-\frac{2^{n-2}}3
\bigl(R_{g_2}-2R_{g_2}\bigr)
\\
&=
\frac{2^{n-2}}3R_{g_2}.
\end{align*}
Indeed, the curvature symmetries give
$R^{g_2}_{kkpp}=0$ and
$R^{g_2}_{pkkp}=-R^{g_2}_{pkpk}$, while
$\sum_{p,k}R^{g_2}_{kpkp}=R_{g_2}$.
\end{proof}

When the principal symbol of $O$ is scalar, these trace
identities simplify \eqref{eq:universal} further.
The resulting formula is the scalar-principal-symbol case of
\cite[Lemma~3.5]{DabrowskiSitarzZalecki2025}.

\begin{proposition}[Reduction formula for the Hodge denominator]
\label{prop:compression}
Let $O$ be a second-order differential operator on $E$
with scalar principal symbol. At $x_0$, write its symbol in
$g_2$-normal coordinates and the coordinate exterior frame as
\[
\sigma(O)(x_0,\xi)
=
f^{ab}\xi_a\xi_b\Id_E+\ii G^a\xi_a+H,
\qquad
f^{ab}=f^{ba}.
\]
Here $f^{ab}$ are the scalar components of the symmetric
contravariant two-tensor determined by the principal symbol.
Its contraction with $g_2$ is $g_{2,ab}f^{ab}$, which equals
$\delta_{ab}f^{ab}$ at $x_0$. Then
\begin{equation}
\label{eq:Hodge-denominator-reduction}
\begin{aligned}
&\int_{|\xi|_{g_2}=1}
\trE\bigl(
\sigma_{-n}(OD_{g_2}^{-n})(x_0,\xi)
\bigr)\dd S_{g_2}(\xi)
=
\nu_{n-1}
\left[
\trE H
-\frac{2^n}{24}R_{g_2}\,g_{2,ab}f^{ab}
\right].
\end{aligned}
\end{equation}
All quantities on the right are evaluated at $x_0$.
If $M$ is closed, integration against $\dd\vol_{g_2}$
gives $\Wres(OD_{g_2}^{-n})$.
\end{proposition}

\begin{proof}
Apply the local form of \eqref{eq:universal} with
$L=D_{g_2}^2$ and $F^{ab}=f^{ab}\Id_E$.
By Lemma~\ref{lem:Hodge-coeff}, $S_a=0$.
The two remaining denominator terms satisfy
\begin{align*}
\trE\bigl(
-12Q+6P_{bb}-2R_{g_2}\Id_E
\bigr)
&=
-12\frac{2^{n-2}}3R_{g_2}
+6\frac{2^n}3R_{g_2}
-2^{n+1}R_{g_2}
\\
&=
-2^nR_{g_2},
\\
\trE\bigl(
-6P_{ab}+2\Ric(g_2)_{ab}\Id_E
\bigr)
&=
-6\frac{2^n}3\Ric(g_2)_{ab}
+2^{n+1}\Ric(g_2)_{ab}
\\
&=
0.
\end{align*}
Since $f^{ab}$ is scalar, it can be taken outside the
fiber trace. Substituting these identities into the local
form of \eqref{eq:universal} gives
\[
\int_{|\xi|_{g_2}=1}
\trE\bigl(
\sigma_{-n}(OD_{g_2}^{-n})(x_0,\xi)
\bigr)\dd S_{g_2}(\xi)
=
\frac{\nu_{n-1}}{24}
\left[
24\trE H
-2^nR_{g_2}\delta_{ab}f^{ab}
\right].
\]
Since $g_{2,ab}(x_0)=\delta_{ab}$, this is
\eqref{eq:Hodge-denominator-reduction}.
\end{proof}

For the numerator $O=D_{g_1}^2$, one has $f^{ab}=g_1^{ab}$.
It remains to compute $\trE H$ in the coordinate exterior
frame associated with $g_2$-normal coordinates.

\subsection{The numerator in \texorpdfstring{$g_2$}{g2}-normal coordinates}
\label{sec:numerator}

We express the numerator $D_{g_1}^2$ in the coordinate exterior
frame associated with $g_2$-normal coordinates. The difference
between the two Levi--Civita connections determines the
connection terms in its symbol.

For smooth vector fields $X,Y$, define
$C(X,Y)=\nabla_X^{g_1}Y-\nabla_X^{g_2}Y$.
The difference of two affine connections is tensorial in both
arguments, so
$C\in\Gamma(T^*M\otimes T^*M\otimes TM)$.
Writing
$C(\partial_i,\partial_j)=C^k{}_{ij}\partial_k$,
we obtain
\begin{equation}
\label{eq:C-components}
C^k{}_{ij}
=
\Gamma(g_1)^k{}_{ij}
-
\Gamma(g_2)^k{}_{ij}.
\end{equation}
We denote its contraction by $c_i=C^k{}_{ik}$.
These components define the one-form $c_i\,dx^i$.

The following form of the Koszul formula expresses the
connection difference in terms of the covariant derivative
of $g_1$ with respect to $\nabla^{g_2}$.

\begin{proposition}[Metric expression for the connection difference]
\label{prop:C-metric}
For all smooth vector fields $X,Y,Z$,
\begin{equation}
\label{eq:C-intrinsic}
\begin{aligned}
2g_1\bigl(C(X,Y),Z\bigr)
={}&
(\nabla_X^{g_2}g_1)(Y,Z)
+
(\nabla_Y^{g_2}g_1)(X,Z)
-
(\nabla_Z^{g_2}g_1)(X,Y).
\end{aligned}
\end{equation}
Consequently, in local coordinates,
\begin{equation}
\label{eq:C-metric}
C^k{}_{ij}
=
\frac12g_1^{k\ell}
\left(
\nabla_i^{g_2}g_{1,j\ell}
+
\nabla_j^{g_2}g_{1,i\ell}
-
\nabla_\ell^{g_2}g_{1,ij}
\right).
\end{equation}
\end{proposition}

\begin{proof}
Since both Levi--Civita connections are torsion free,
\begin{align*}
C(X,Y)-C(Y,X)
&=
\bigl(\nabla_X^{g_1}Y-\nabla_Y^{g_1}X\bigr)
-
\bigl(\nabla_X^{g_2}Y-\nabla_Y^{g_2}X\bigr)
\\
&=
[X,Y]-[X,Y]
=
0.
\end{align*}
Thus $C(X,Y)=C(Y,X)$.

Metric compatibility, $\nabla^{g_1}g_1=0$, gives
\[
X\bigl(g_1(Y,Z)\bigr)
=
g_1(\nabla_X^{g_1}Y,Z)
+
g_1(Y,\nabla_X^{g_1}Z).
\]
Substituting
$\nabla_X^{g_1}Y=\nabla_X^{g_2}Y+C(X,Y)$
and the corresponding identity for $Z$, we obtain
\begin{align*}
X\bigl(g_1(Y,Z)\bigr)
&=
g_1\bigl(\nabla_X^{g_2}Y+C(X,Y),Z\bigr)
+
g_1\bigl(Y,\nabla_X^{g_2}Z+C(X,Z)\bigr)
\\
&=
g_1(\nabla_X^{g_2}Y,Z)
+
g_1(Y,\nabla_X^{g_2}Z)
+
g_1(C(X,Y),Z)
+
g_1(Y,C(X,Z)).
\end{align*}
By the definition of the covariant derivative of a
covariant two-tensor,
\[
(\nabla_X^{g_2}g_1)(Y,Z)
=
X\bigl(g_1(Y,Z)\bigr)
-
g_1(\nabla_X^{g_2}Y,Z)
-
g_1(Y,\nabla_X^{g_2}Z).
\]
Combining these identities and using the symmetry of $g_1$
therefore yields
\[
(\nabla_X^{g_2}g_1)(Y,Z)
=
g_1(C(X,Y),Z)
+
g_1(C(X,Z),Y).
\]
Applying the same identity with differentiating vector fields
$Y$ and $Z$, respectively, gives
\begin{align*}
(\nabla_Y^{g_2}g_1)(X,Z)
&=
g_1(C(Y,X),Z)
+
g_1(C(Y,Z),X),
\\
(\nabla_Z^{g_2}g_1)(X,Y)
&=
g_1(C(Z,X),Y)
+
g_1(C(Z,Y),X).
\end{align*}
Adding the first two identities and subtracting the third,
we find
\begin{align*}
&(\nabla_X^{g_2}g_1)(Y,Z)
+
(\nabla_Y^{g_2}g_1)(X,Z)
-
(\nabla_Z^{g_2}g_1)(X,Y)
\\
=&
g_1\bigl(C(X,Y)+C(Y,X),Z\bigr)
+
g_1\bigl(C(X,Z)-C(Z,X),Y\bigr)
+
g_1\bigl(C(Y,Z)-C(Z,Y),X\bigr)
\\
=&
2g_1(C(X,Y),Z),
\end{align*}
where the last equality follows from the symmetry of $C$.
This proves \eqref{eq:C-intrinsic}.

To obtain the coordinate formula, take
$X=\partial_i$, $Y=\partial_j$, and $Z=\partial_\ell$.
Since
$C(\partial_i,\partial_j)=C^k{}_{ij}\partial_k$,
the resulting identity is
\[
2g_{1,k\ell}C^k{}_{ij}
=
\nabla_i^{g_2}g_{1,j\ell}
+
\nabla_j^{g_2}g_{1,i\ell}
-
\nabla_\ell^{g_2}g_{1,ij}.
\]
Here
$\nabla_i^{g_2}g_{1,j\ell}
=(\nabla_{\partial_i}^{g_2}g_1)(\partial_j,\partial_\ell)$;
explicitly,
\[
\nabla_i^{g_2}g_{1,j\ell}
=
\partial_i g_{1,j\ell}
-
\Gamma(g_2)^p{}_{ij}g_{1,p\ell}
-
\Gamma(g_2)^p{}_{i\ell}g_{1,jp}.
\]
Contracting with $g_1^{r\ell}$ and using
$g_1^{r\ell}g_{1,k\ell}=\delta^r{}_k$, we obtain
\[
2C^r{}_{ij}
=
g_1^{r\ell}
\left(
\nabla_i^{g_2}g_{1,j\ell}
+
\nabla_j^{g_2}g_{1,i\ell}
-
\nabla_\ell^{g_2}g_{1,ij}
\right).
\]
Dividing by two and renaming the free index $r$ as $k$
gives \eqref{eq:C-metric}.
\end{proof}

We next record the contraction of $C$ and the corresponding
identity for the inverse metric.

\begin{lemma}[Elementary identities for the connection difference]
\label{lem:C-identities}
The contracted connection difference and the inverse metric
satisfy
\begin{align}
c_i=C^k{}_{ik}
&=
\frac12g_1^{ab}\nabla_i^{g_2}g_{1,ab}
,
\label{eq:c-density}
\\
\nabla_i^{g_2}g_1^{jk}
&=
-C^j{}_{i\ell}g_1^{\ell k}
-C^k{}_{i\ell}g_1^{j\ell}.
\label{eq:inverse-difference}
\end{align}
\end{lemma}

\begin{proof}
Contracting the upper index with the second lower index
in \eqref{eq:C-metric} gives
\begin{align*}
c_i
&=
\frac12g_1^{k\ell}
\left(
\nabla_i^{g_2}g_{1,k\ell}
+
\nabla_k^{g_2}g_{1,i\ell}
-
\nabla_\ell^{g_2}g_{1,ik}
\right)
\\
&=
\frac12g_1^{k\ell}\nabla_i^{g_2}g_{1,k\ell}.
\end{align*}
The last two terms in the first line cancel after
interchanging the dummy indices $k$ and $\ell$.

The covariant derivative of this scalar function is its
ordinary derivative, which proves \eqref{eq:c-density}.

For the inverse metric, the induced connection acts on
each contravariant index. Hence
\[
0
=
\nabla_i^{g_1}g_1^{jk}
=
\nabla_i^{g_2}g_1^{jk}
+
C^j{}_{i\ell}g_1^{\ell k}
+
C^k{}_{i\ell}g_1^{j\ell},
\]
where the first equality follows from
$\nabla^{g_1}g_1^{-1}=0$.
Rearranging proves \eqref{eq:inverse-difference}.
\end{proof}

We now describe the induced connection on the exterior
bundle in the coordinate frame.

\begin{lemma}[The induced exterior connection in normal coordinates]
\label{lem:exterior-connection}
Choose $g_2$-normal coordinates centered at $x_0$.
Let $\nabla^{\Lambda,g_1}$ be the connection on $E$ induced
by $\nabla^{g_1}$, and write
$\nabla_i^{\Lambda,g_1}
=\nabla_{\partial_i}^{\Lambda,g_1}$.
Define its connection matrix in the coordinate exterior
frame by
\[
\nabla_i^{\Lambda,g_1}
=
\partial_i+\omega_i^{(1)},
\]
where $\partial_i$ differentiates only the coefficient
functions in this frame.
Throughout the coordinate neighborhood,
\begin{equation}
\label{eq:omega}
\omega_i^{(1)}
=
-\Gamma(g_1)^a{}_{ib}
\eps(dx^b)\iotaop(\partial_a).
\end{equation}
At the center $x_0$, one has
$\Gamma(g_2)^a{}_{ib}(x_0)=0$,
$\Gamma(g_1)^a{}_{ib}(x_0)=C^a{}_{ib}(x_0)$, and therefore
\[
\omega_i^{(1)}(x_0)
=
-C^a{}_{ib}(x_0)
\bigl[\eps(dx^b)\iotaop(\partial_a)\bigr]_{x_0}.
\]
\end{lemma}

\begin{proof}
The connection on one-forms is the dual of the connection
on vector fields. Since $dx^c(\partial_a)=\delta^c{}_a$
is constant,
\begin{align*}
0
&=
\partial_i\bigl(dx^c(\partial_a)\bigr)
\\
&=
(\nabla_{\partial_i}^{g_1}dx^c)(\partial_a)
+
dx^c(\nabla_{\partial_i}^{g_1}\partial_a)
\\
&=
(\nabla_{\partial_i}^{g_1}dx^c)(\partial_a)
+
\Gamma(g_1)^c{}_{ia}.
\end{align*}
Thus
$\nabla_{\partial_i}^{g_1}dx^c
=-\Gamma(g_1)^c{}_{ib}dx^b$.
The endomorphism in \eqref{eq:omega} has precisely this
action on a coordinate one-form:
\begin{align*}
-\Gamma(g_1)^a{}_{ib}
\eps(dx^b)\iotaop(\partial_a)dx^c
&=
-\Gamma(g_1)^a{}_{ib}\delta^c{}_a\,dx^b
\\
&=
-\Gamma(g_1)^c{}_{ib}dx^b.
\end{align*}

To verify its action in every degree, let
$I=(j_1,\ldots,j_p)$ with
$1\leq j_1<\cdots<j_p\leq n$, and put
$dx^I=dx^{j_1}\wedge\cdots\wedge dx^{j_p}$.
The Leibniz rule for the induced connection gives
\begin{align*}
\nabla_i^{\Lambda,g_1}dx^I
&=
\sum_{s=1}^p
dx^{j_1}\wedge\cdots\wedge
\bigl(\nabla_{\partial_i}^{g_1}dx^{j_s}\bigr)
\wedge\cdots\wedge dx^{j_p}
\\
&=
-\sum_{s=1}^p\sum_{b=1}^n
\Gamma(g_1)^{j_s}{}_{ib}\,
dx^{j_1}\wedge\cdots\wedge dx^b
\wedge\cdots\wedge dx^{j_p},
\end{align*}
where $dx^b$ occupies the $s$th position.

On the other hand, contraction satisfies
\[
\iotaop(\partial_a)dx^I
=
\sum_{s=1}^p
(-1)^{s-1}\delta^{j_s}{}_a\,
dx^{j_1}\wedge\cdots\wedge
\widehat{dx^{j_s}}\wedge\cdots\wedge dx^{j_p}.
\]
The hat denotes omission of that factor.
Exterior multiplication by $dx^b$ places it first:
\[
\eps(dx^b)\iotaop(\partial_a)dx^I
=
\sum_{s=1}^p
(-1)^{s-1}\delta^{j_s}{}_a\,
dx^b\wedge dx^{j_1}\wedge\cdots\wedge
\widehat{dx^{j_s}}\wedge\cdots\wedge dx^{j_p}.
\]
Moving $dx^b$ to the $s$th position requires $s-1$
interchanges and introduces another factor $(-1)^{s-1}$.
Since $(-1)^{s-1}(-1)^{s-1}=1$, this becomes
\[
\eps(dx^b)\iotaop(\partial_a)dx^I
=
\sum_{s=1}^p
\delta^{j_s}{}_a\,
dx^{j_1}\wedge\cdots\wedge dx^b
\wedge\cdots\wedge dx^{j_p}.
\]
Consequently,
\begin{align*}
-\Gamma(g_1)^a{}_{ib}
\eps(dx^b)\iotaop(\partial_a)dx^I
&
=
-\sum_{s=1}^p\sum_{b=1}^n
\Gamma(g_1)^{j_s}{}_{ib}\,
dx^{j_1}\wedge\cdots\wedge dx^b
\wedge\cdots\wedge dx^{j_p}
\\
&
=
\nabla_i^{\Lambda,g_1}dx^I.
\end{align*}
In degree zero, we set $dx^\varnothing=1$.
Both $\nabla_i^{\Lambda,g_1}1$ and
$\iotaop(\partial_a)1$ vanish.

For a general section $\alpha=\sum_I\alpha_I dx^I$,
where the sum runs over all exterior degrees, the Leibniz
rule therefore gives
\begin{align*}
\nabla_i^{\Lambda,g_1}\alpha
&=
\sum_I(\partial_i\alpha_I)dx^I
+
\sum_I\alpha_I\nabla_i^{\Lambda,g_1}dx^I
\\
&=
\partial_i\alpha
-
\Gamma(g_1)^a{}_{ib}
\eps(dx^b)\iotaop(\partial_a)\alpha.
\end{align*}
Here $\partial_i\alpha=\sum_I(\partial_i\alpha_I)dx^I$.
Comparison with
$\nabla_i^{\Lambda,g_1}=\partial_i+\omega_i^{(1)}$
proves \eqref{eq:omega}.

Finally, $g_2$-normal coordinates give
$\Gamma(g_2)^a{}_{ib}(x_0)=0$.
Equation~\eqref{eq:C-components} then yields
$\Gamma(g_1)^a{}_{ib}(x_0)=C^a{}_{ib}(x_0)$,
which proves the asserted formula at $x_0$.
\end{proof}

At the center of these normal coordinates,
Lemma~\ref{lem:exterior-connection} and the trace identities
\eqref{eq:lambda2} and \eqref{eq:lambda4} give
\begin{align*}
\trE\omega_i^{(1)}
&=
-2^{n-1}C^a{}_{ib}\delta^b{}_a
=
-2^{n-1}c_i,
\\
\trE\bigl(\omega_i^{(1)}\omega_j^{(1)}\bigr)
&=
2^{n-2}C^a{}_{ib}C^c{}_{jd}
\bigl(
\delta^b{}_a\delta^d{}_c
+
\delta^b{}_c\delta^d{}_a
\bigr)
\\
&=
2^{n-2}
\bigl(
c_ic_j+C^k{}_{i\ell}C^\ell{}_{jk}
\bigr).
\end{align*}

Let $\mathcal R_{g_1}$ denote the curvature endomorphism
in the Weitzenb\"ock formula for the Hodge Laplacian.
The standard decomposition \cite{Gilkey1995} reads
\begin{equation}
\label{eq:Weitzenbock}
D_{g_1}^2
=
-g_1^{ij}
\left(
\nabla_i^{\Lambda,g_1}\nabla_j^{\Lambda,g_1}
-
\Gamma(g_1)^k{}_{ij}\nabla_k^{\Lambda,g_1}
\right)
+
\mathcal R_{g_1}.
\end{equation}

We compute the fiber trace of $\mathcal R_{g_1}$ from
its action on forms. At a point, choose a $g_1$-orthonormal
basis $e_1,\ldots,e_n$ with dual basis $e^1,\ldots,e^n$.
For an increasing $p$-tuple $I$, write
$e^I=e^{i_1}\wedge\cdots\wedge e^{i_p}$.
With respect to the induced inner product on forms, the
diagonal entry is
\[
\bigl\langle\mathcal R_{g_1}e^I,e^I\bigr\rangle_{g_1}
=
\sum_{\substack{i\in I\\j\notin I}}
g_1\bigl(R^{g_1}(e_i,e_j)e_j,e_i\bigr).
\]
For a fixed unordered pair $\{i,j\}$, there are
$\binom{n-2}{p-1}$ subsets $I$ containing $i$ but not $j$,
and the same number containing $j$ but not $i$.
Since
$R_{g_1}
=2\sum_{i<j}g_1(R^{g_1}(e_i,e_j)e_j,e_i)$,
it follows that, for $1\leq p\leq n-1$,
\begin{align*}
\operatorname{tr}_{\Lambda^pT^*M\otimes\mathbb C}
\mathcal R_{g_1}
&=
2\binom{n-2}{p-1}
\sum_{i<j}
g_1\bigl(R^{g_1}(e_i,e_j)e_j,e_i\bigr)
\\
&=
\binom{n-2}{p-1}R_{g_1}.
\end{align*}
The curvature action is zero in degrees $0$ and $n$.
Summing over all degrees therefore gives
\begin{equation}
\label{eq:trace-curvature}
\trE\mathcal R_{g_1}
=
\left(
\sum_{p=1}^{n-1}\binom{n-2}{p-1}
\right)R_{g_1}
=
2^{n-2}R_{g_1}.
\end{equation}

We can now determine the numerator symbol.
For a section $\alpha$ expressed in the coordinate exterior
frame, the product rule gives
\begin{align*}
\nabla_i^{\Lambda,g_1}\nabla_j^{\Lambda,g_1}\alpha
&=
(\partial_i+\omega_i^{(1)})
(\partial_j\alpha+\omega_j^{(1)}\alpha)
\\
&=
\partial_i\partial_j\alpha
+
\omega_j^{(1)}\partial_i\alpha
+
\omega_i^{(1)}\partial_j\alpha
+
\bigl(
\partial_i\omega_j^{(1)}
+
\omega_i^{(1)}\omega_j^{(1)}
\bigr)\alpha.
\end{align*}
Because $g_1^{ij}$ is symmetric, contraction of the two
first-order terms gives
$2g_1^{ij}\omega_i^{(1)}\partial_j\alpha$.
Substitution into \eqref{eq:Weitzenbock} yields
\begin{align*}
D_{g_1}^2\alpha
={}&
-g_1^{ij}\partial_i\partial_j\alpha
\\
&+
\left(
-2g_1^{ik}\omega_i^{(1)}
+
g_1^{ij}\Gamma(g_1)^k{}_{ij}\Id_E
\right)\partial_k\alpha
\\
&+
\left[
-g_1^{ij}
\left(
\partial_i\omega_j^{(1)}
+
\omega_i^{(1)}\omega_j^{(1)}
-
\Gamma(g_1)^k{}_{ij}\omega_k^{(1)}
\right)
+
\mathcal R_{g_1}
\right]\alpha.
\end{align*}
Using $\sigma(\partial_k)=\ii\xi_k$, we therefore obtain
$\sigma(D_{g_1}^2)
=F^{ab}\xi_a\xi_b+\ii G^a\xi_a+H$, where
\begin{align*}
F^{ab}
&=
g_1^{ab}\Id_E,
\\
G^k
&=
-2g_1^{ik}\omega_i^{(1)}
+
g_1^{ij}\Gamma(g_1)^k{}_{ij}\Id_E,
\\
H
&=
-g_1^{ij}
\left(
\partial_i\omega_j^{(1)}
+
\omega_i^{(1)}\omega_j^{(1)}
-
\Gamma(g_1)^k{}_{ij}\omega_k^{(1)}
\right)
+
\mathcal R_{g_1}.
\end{align*}
In particular, the principal symbol is scalar, so
Proposition~\ref{prop:compression} applies with
$f^{ab}=g_1^{ab}$. It remains to evaluate $\trE H$.

\begin{proposition}[Trace of the zero-order numerator symbol]
\label{prop:trH}
Let $H=\sigma_0(D_{g_1}^2)$ be computed in the coordinate
exterior frame used above. At the center $x_0$ of the
$g_2$-normal coordinates,
\begin{equation}
\label{eq:trH}
\begin{aligned}
\trE H
={}&
2^{n-2}R_{g_1}
+
2^{n-1}g_1^{ij}\nabla_i^{g_2}c_j
-
2^{n-2}g_1^{ij}
\bigl(
c_ic_j+C^k{}_{i\ell}C^\ell{}_{jk}
\bigr)
\\
&-
2^{n-1}g_1^{ij}C^r{}_{ij}c_r
-
\frac{2^{n-1}}3g_1^{ij}\Ric(g_2)_{ij}.
\end{aligned}
\end{equation}
\end{proposition}

\begin{proof}
We first compute the derivative term in $H$.
Equation~\eqref{eq:omega} holds throughout the coordinate
neighborhood, and \eqref{eq:lambda2} gives
$\trE\omega_j^{(1)}
=-2^{n-1}\Gamma(g_1)^a{}_{ja}$.
Differentiating this identity before evaluating at $x_0$,
we obtain
\begin{align*}
\trE\bigl(\partial_i\omega_j^{(1)}\bigr)
&=
-2^{n-1}\partial_i\Gamma(g_1)^a{}_{ja}
\\
&=
-2^{n-1}
\left(
\partial_i\Gamma(g_2)^a{}_{ja}
+
\partial_i c_j
\right),
\end{align*}
where the second equality follows from
\eqref{eq:C-components}.

At the normal-coordinate center,
$\partial_i c_j=\nabla_i^{g_2}c_j$.
Moreover, \eqref{eq:contracted-Gamma} gives
$\partial_i\Gamma(g_2)^a{}_{ja}
=-\Ric(g_2)_{ij}/3$.
Consequently,
\[
\trE\bigl(\partial_i\omega_j^{(1)}\bigr)
=
-2^{n-1}\nabla_i^{g_2}c_j
+
\frac{2^{n-1}}3\Ric(g_2)_{ij}.
\]

The derivative term in the trace of $H$ is thus
\[
-g_1^{ij}\trE\bigl(\partial_i\omega_j^{(1)}\bigr)
=
2^{n-1}g_1^{ij}\nabla_i^{g_2}c_j
-
\frac{2^{n-1}}3g_1^{ij}\Ric(g_2)_{ij}.
\]
The quadratic connection term, computed above from
\eqref{eq:lambda4}, is
\[
-g_1^{ij}
\trE\bigl(\omega_i^{(1)}\omega_j^{(1)}\bigr)
=
-2^{n-2}g_1^{ij}
\bigl(
c_ic_j+C^k{}_{i\ell}C^\ell{}_{jk}
\bigr).
\]
For the Christoffel correction, use
$\Gamma(g_1)^r{}_{ij}=C^r{}_{ij}$ at $x_0$ and
$\trE\omega_r^{(1)}=-2^{n-1}c_r$ to obtain
\[
g_1^{ij}\Gamma(g_1)^r{}_{ij}\trE\omega_r^{(1)}
=
-2^{n-1}g_1^{ij}C^r{}_{ij}c_r.
\]
Finally,
$\trE\mathcal R_{g_1}=2^{n-2}R_{g_1}$
by \eqref{eq:trace-curvature}.
Substituting these contributions into
\begin{align*}
\trE H
={}&
-g_1^{ij}\trE\bigl(\partial_i\omega_j^{(1)}\bigr)
-
g_1^{ij}\trE\bigl(\omega_i^{(1)}\omega_j^{(1)}\bigr)
\\
&+
g_1^{ij}\Gamma(g_1)^r{}_{ij}\trE\omega_r^{(1)}
+
\trE\mathcal R_{g_1}
\end{align*}
proves \eqref{eq:trH}.
\end{proof}

\subsection{The closed bimetric KKW theorem}
\label{sec:closed}

Define the local bimetric scalar by
\begin{align}
\mathscr L_n(g_1\mid g_2)
:={}&
R_{g_1}
+
2g_1^{ij}\nabla_i^{g_2}c_j
-
g_1^{ij}
\bigl(
c_ic_j+C^k{}_{i\ell}C^\ell{}_{jk}
\bigr)
\notag\\
&-
2g_1^{ij}C^r{}_{ij}c_r
-
\frac23g_1^{ij}\Ric(g_2)_{ij}
-
\frac16R_{g_2}\,g_{2,ij}g_1^{ij}.
\label{eq:local-density}
\end{align}
Here $C$ is the connection difference defined in
\eqref{eq:C-components}, and $c_i=C^k{}_{ik}$ is equivalently
expressed by \eqref{eq:c-density}. Every term in
\eqref{eq:local-density} is a complete tensor contraction,
so $\mathscr L_n(g_1\mid g_2)$ is a globally defined smooth
function.

\begin{theorem}[Local-density form of the closed bimetric KKW formula]
\label{thm:local-closed}
Let $M^n$ be a closed oriented manifold of even dimension
$n=2m\geq2$, and let $g_1,g_2$ be Riemannian metrics on $M$.
Then
\[
\Wres(D_{g_1}^2D_{g_2}^{-n})
=
2^{n-2}\nu_{n-1}
\int_M
\mathscr L_n(g_1\mid g_2)\dd\vol_{g_2},
\]
where $\mathscr L_n(g_1\mid g_2)$ is defined by
\eqref{eq:local-density}.
\end{theorem}

\begin{proof}
At an arbitrary point $x_0$, use $g_2$-normal coordinates
and their coordinate exterior frame.
The scalar principal coefficient of $D_{g_1}^2$ is
$f^{ab}=g_1^{ab}$.
By Proposition~\ref{prop:trH} and the identity
$2^n/24=2^{n-2}/6$,
\[
\trE H
-
\frac{2^n}{24}R_{g_2}\,g_{2,ab}g_1^{ab}
=
2^{n-2}\mathscr L_n(g_1\mid g_2).
\]
Substituting into
\eqref{eq:Hodge-denominator-reduction} gives the pointwise
identity
\[
\int_{|\xi|_{g_2}=1}
\trE\bigl(
\sigma_{-n}(D_{g_1}^2D_{g_2}^{-n})(x_0,\xi)
\bigr)\dd S_{g_2}(\xi)
=
2^{n-2}\nu_{n-1}\mathscr L_n(g_1\mid g_2)(x_0).
\]
Integration against $\dd\vol_{g_2}$ proves the theorem.
\end{proof}

The derivative term in \eqref{eq:local-density} can be
written as a divergence and algebraic contractions.
On a closed manifold, this gives the following equivalent
integral formula.

\begin{theorem}[Global closed formula]
\label{thm:global-closed}
Under the hypotheses of Theorem~\ref{thm:local-closed},
\begin{equation}
\label{eq:global-main}
\begin{aligned}
\Wres(D_{g_1}^2D_{g_2}^{-n})
={}&
2^{n-2}\nu_{n-1}
\int_M
\Bigl[
R_{g_1}
+
g_1^{ij}
\bigl(
c_ic_j-C^k{}_{i\ell}C^\ell{}_{jk}
\bigr)
\\
&\qquad
-
\frac23g_1^{ij}\Ric(g_2)_{ij}
-
\frac16R_{g_2}\,g_{2,ij}g_1^{ij}
\Bigr]\dd\vol_{g_2}.
\end{aligned}
\end{equation}
\end{theorem}

\begin{proof}
Contracting \eqref{eq:inverse-difference} gives
\begin{align*}
\nabla_i^{g_2}g_1^{ij}
&=
-C^i{}_{ik}g_1^{kj}
-
C^j{}_{ik}g_1^{ik}
\\
&=
-c_kg_1^{kj}
-
C^j{}_{ik}g_1^{ik}.
\end{align*}
The second equality uses
$C^i{}_{ik}=C^i{}_{ki}=c_k$.
The product rule therefore yields
\begin{align*}
\nabla_i^{g_2}(g_1^{ij}c_j)
&=
(\nabla_i^{g_2}g_1^{ij})c_j
+
g_1^{ij}\nabla_i^{g_2}c_j
\\
&=
-g_1^{ij}c_ic_j
-
g_1^{ij}C^r{}_{ij}c_r
+
g_1^{ij}\nabla_i^{g_2}c_j,
\end{align*}
where dummy indices have been renamed in the second line.
Consequently,
\begin{align*}
&2g_1^{ij}\nabla_i^{g_2}c_j
-
g_1^{ij}c_ic_j
-
2g_1^{ij}C^r{}_{ij}c_r
=
2\nabla_i^{g_2}(g_1^{ij}c_j)
+
g_1^{ij}c_ic_j.
\end{align*}
Substituting this identity into \eqref{eq:local-density},
we obtain
\begin{align*}
\mathscr L_n(g_1\mid g_2)
={}&
R_{g_1}
+
g_1^{ij}
\bigl(
c_ic_j-C^k{}_{i\ell}C^\ell{}_{jk}
\bigr)
-
\frac23g_1^{ij}\Ric(g_2)_{ij}
-
\frac16R_{g_2}\,g_{2,ij}g_1^{ij}
+
2\nabla_i^{g_2}(g_1^{ij}c_j).
\end{align*}

The last term is twice the $g_2$-divergence of the globally
defined vector field $g_1^{ij}c_j\partial_i$.
Since $M$ is closed, the divergence theorem gives
\[
\int_M
\nabla_i^{g_2}(g_1^{ij}c_j)\dd\vol_{g_2}
=
0.
\]
Integrating the preceding expression for
$\mathscr L_n(g_1\mid g_2)$ and applying
Theorem~\ref{thm:local-closed} proves
\eqref{eq:global-main}.
\end{proof}

\begin{corollary}[Diagonal limit]
\label{cor:diagonal}
If $g_1=g_2=g$, then
\[
\Wres(D_g^{-n+2})
=
-\frac{(n-2)2^n}{24}\nu_{n-1}
\int_M R_g\dd\vol_g.
\]
\end{corollary}

\begin{proof}
When the two metrics coincide, $C=0$ and $c=0$.
Moreover,
$g^{ij}\Ric(g)_{ij}=R_g$ and $g_{ij}g^{ij}=n$.
Thus
\[
\mathscr L_n(g\mid g)
=
R_g-\frac23R_g-\frac n6R_g
=
-\frac{n-2}{6}R_g.
\]
The parametrix identities give
$D_g^2D_g^{-n}=D_g^{-n+2}$ modulo smoothing operators,
which have zero noncommutative residue.
Theorem~\ref{thm:local-closed} therefore yields the stated
formula, since
$2^{n-2}(n-2)/6=(n-2)2^n/24$.
For $n=2$, it reduces to $\Wres(\Id_E)=0$.
\end{proof}

For brevity, write
$\mathcal K_n(g_1\mid g_2)
:=\Wres(D_{g_1}^2D_{g_2}^{-n})$.
The metric $g_1$ determines the differential factor,
whereas $g_2$ determines the elliptic negative power and
the volume form used to express the local residue density.

\begin{corollary}[Constant scaling]
\label{cor:scaling}
For a constant $\lambda>0$,
\[
\mathcal K_n(\lambda^2g_2\mid g_2)
=
\lambda^{-2}\Wres(D_{g_2}^{-n+2}).
\]
\end{corollary}

\begin{proof}
Under a constant rescaling of the metric,
$\delta_{\lambda^2g_2}=\lambda^{-2}\delta_{g_2}$.
Since $d^2=0$ and $\delta_{g_2}^2=0$, it follows that
\begin{align*}
D_{\lambda^2g_2}^2
&=
(d+\lambda^{-2}\delta_{g_2})^2
\\
&=
\lambda^{-2}
(d\delta_{g_2}+\delta_{g_2}d)
\\
&=
\lambda^{-2}D_{g_2}^2.
\end{align*}
By linearity of the noncommutative residue and the
parametrix identities,
\begin{align*}
\mathcal K_n(\lambda^2g_2\mid g_2)
&=
\Wres\bigl(
\lambda^{-2}D_{g_2}^2D_{g_2}^{-n}
\bigr)
\\
&=
\lambda^{-2}\Wres(D_{g_2}^{-n+2}),
\end{align*}
as claimed.
\end{proof}

The scaling relations also exhibit the asymmetry of the two
metrics. Indeed, because $n=2m$, the identity
$D_{\lambda^2g}^2=\lambda^{-2}D_g^2$ implies
\[
D_{\lambda^2g}^{-n}
=
(D_{\lambda^2g}^2)^{-m}
=
\lambda^{2m}(D_g^2)^{-m}
=
\lambda^nD_g^{-n}
\]
modulo smoothing operators. Hence
\[
\mathcal K_n(g\mid\lambda^2g)
=
\lambda^n\Wres(D_g^{-n+2}),
\]
whereas Corollary~\ref{cor:scaling} gives
$\mathcal K_n(\lambda^2g\mid g)
=\lambda^{-2}\Wres(D_g^{-n+2})$.
If $n\geq4$, $\lambda\neq1$, and
$\int_M R_g\dd\vol_g\neq0$,
Corollary~\ref{cor:diagonal} shows that these two values
are different. Thus $\mathcal K_n$ need not be symmetric
in its two metric arguments.

\begin{remark}[Pseudo-Riemannian numerator]
\label{rem:pseudo-Riemannian-numerator}
Theorems~\ref{thm:local-closed} and
\ref{thm:global-closed} remain valid when $g_1$ is a smooth
pseudo-Riemannian metric, including a Lorentzian metric,
while $g_2$ remains Riemannian. In this case, define
\[
\delta_{g_1}
=
-g_1^{ij}\iotaop(\partial_i)
\nabla_{\partial_j}^{\Lambda,g_1},
\qquad
D_{g_1}=d+\delta_{g_1}.
\]
The square $D_{g_1}^2$ has scalar principal symbol
$\sigma_2(D_{g_1}^2)(x,\xi)
=g_1^{ij}(x)\xi_i\xi_j\Id_E$.
Proposition~\ref{prop:compression} requires no ellipticity
of the inserted operator. Since only the elliptic operator
$D_{g_2}$ is inverted, $D_{g_1}^2D_{g_2}^{-n}$ remains
a classical pseudodifferential operator.

The numerator identities, including
\eqref{eq:trace-curvature} and \eqref{eq:trH}, also extend
to arbitrary nondegenerate signature. In the coordinate
exterior frame, their coefficients are rational expressions
in the metric components and their derivatives up to
order two, with denominators given by powers of
$\det(g_{1,ij})$. Clearing denominators gives polynomial
identities. Their validity for positive-definite metric
values and arbitrary first and second derivatives therefore
implies their validity for every nondegenerate metric.
\end{remark}

We recall the spectral interpretation of the preceding formula.
A unital spectral triple $(\mathcal A,\mathcal H,D)$ consists
of a unital $*$-algebra $\mathcal A$ faithfully represented
by bounded operators on a Hilbert space $\mathcal H$ and
a densely defined self-adjoint operator $D$ with compact
resolvent, such that each $a\in\mathcal A$ preserves
$\operatorname{Dom}(D)$ and $[D,a]$ extends to a bounded
operator; see \cite{Connes1994}.

In the classical setting considered here, $D$ is a
first-order elliptic differential operator on a vector
bundle over a closed manifold of dimension $n=2m$.
For a classical pseudodifferential operator $O$ on the
same bundle, the noncommutative integral relative to $D$
is the linear functional
\[
O\longmapsto \Wres(OD^{-n}).
\]
Here $D^{-n}=(D^2)^{-m}$ is taken on
$(\ker D)^\perp$ and is set equal to zero on $\ker D$.
Equivalently, one may use a parametrix, since smoothing
operators have zero residue.

For two spectral triples $(\mathcal A,\mathcal H,D_1)$
and $(\mathcal A,\mathcal H,D_2)$ whose operators belong
to the same classical pseudodifferential calculus,
the double spectral Einstein--Hilbert action is
\[
\operatorname{EH}(D_1,D_2)
:=
\Wres(D_1^2D_2^{-n}),
\]
as in \cite[Definition~4.1]{LiuWang2026}.
On the diagonal this gives
$\operatorname{EH}(D,D)=\Wres(D^{-n+2})$.

The corresponding scalar construction uses the
Laplace--Beltrami operators $\Delta^{g_1}$ and
$\Delta^{g_2}$ acting on functions.
For a smooth pseudo-Riemannian metric $g_1$ and a
Riemannian metric $g_2$, the bimetric spectral
Einstein--Hilbert action is defined by
\[
\operatorname{EH}(\Delta^{g_1},\Delta^{g_2})
:=
\Wres\bigl(\Delta^{g_1}(\Delta^{g_2})^{-m}\bigr);
\]
see \cite[Definition~4.2]{LiuWang2026}.
Only the elliptic operator $\Delta^{g_2}$ is inverted.

For the exterior bundle $E$, the reference spectral
triple is
$(C^\infty(M;\mathbb C),L^2(M,E;g_2),D_{g_2})$,
where functions act by multiplication and $D_{g_2}$
denotes its self-adjoint closure.
The noncommutative integral of the differential
insertion $D_{g_1}^2$ is therefore
$\Wres(D_{g_1}^2D_{g_2}^{-n})$.

\begin{definition}[Bimetric Hodge spectral Einstein--Hilbert action]
Let $M^n$ be closed and oriented, with $n=2m$,
and let $g_1$ be a smooth pseudo-Riemannian metric
and $g_2$ a Riemannian metric.
We define the bimetric Hodge spectral Einstein--Hilbert
action by the right-hand side of
\eqref{eq:global-main}.
By Theorem~\ref{thm:global-closed} and
Remark~\ref{rem:pseudo-Riemannian-numerator}, it equals
\[
\mathcal K_n(g_1\mid g_2)
=
\Wres(D_{g_1}^2D_{g_2}^{-n}).
\]
\end{definition}

In this construction, $D_{g_1}^2$ is a differential
insertion; $D_{g_1}$ need not define a second spectral
triple on $L^2(M,E;g_2)$.
The diagonal specialization is given by
Corollary~\ref{cor:diagonal}.

\section{The noncommutative residue for manifolds with boundary}
\label{sec:boundary-framework}

We now consider a compact oriented manifold with smooth boundary,
of even dimension $n=2m\ge4$. We introduce the collar metrics and
the local boundary residue formula, and derive the symbol and
trace identities needed for the two factorizations. All operators
are represented in the same coordinate exterior frame.

\subsection{Boundary geometry and the residue formula}
\label{sec:FGLS-framework}

Let $g^{\partial M}$ be a Riemannian metric on $\partial M$. We assume
that a collar $\partial M\times[0,\varepsilon)$ has coordinates
$(x',x_n)$, where $x_n$ denotes the last coordinate $x^n$ and
increases into $M$, and that
\begin{equation}
\label{eq:collar}
g_r=h_r(x_n)^{-1}g^{\partial M}+dx_n^2,
\qquad
h_r>0,\quad h_r(0)=1,\quad r=1,2.
\end{equation}
Here $g^{\partial M}$ is pulled back from $\partial M$, and the functions
$h_r$ are smooth on $[0,\varepsilon)$. They may be chosen separately
on different boundary components. Both metrics are expressed
in the same collar coordinates and induce the same boundary
metric $g^{\partial M}$. This is the two-metric version of the collar
assumption in \cite[(1.3)]{Wang2007}.

Tangential indices $a,b$ range from $1$ to $n-1$, and
$\partial_n=\partial/\partial x_n$ is the common inward unit
normal. For vectors $X,Y$ tangent to $\partial M$, define
\begin{equation}
\label{eq:mean-curvature}
\mathrm{II}_{g_r}(X,Y)
=g_r(\nabla_X^{g_r}\partial_n,Y),
\qquad
K_{g_r}=\operatorname{tr}_{g^{\partial M}}\mathrm{II}_{g_r}.
\end{equation}
Thus $K_{g_r}$ is the trace, rather than the average, of the
principal curvatures. Since $g_{r,an}=0$ and $g_{r,nn}=1$,
the Koszul formula gives
\[
2g_r(\nabla_{\partial_a}^{g_r}\partial_n,\partial_b)
=\partial_n g_{r,ab}.
\]
Using $\partial_n(h_r^{-1})|_0=-h_r'(0)$, we obtain
\begin{equation}
\label{eq:warped-mean-curvature}
\mathrm{II}_{g_r}=-\frac{h_r'(0)}2g^{\partial M},
\qquad
K_{g_r}=-\frac{n-1}{2}h_r'(0).
\end{equation}
In particular,
$K_{g_1}-K_{g_2}=-(n-1)(h_1'(0)-h_2'(0))/2$.

Extend both metrics smoothly across $\partial M$ to a closed
manifold containing $M$. All negative powers below are
parametrices on this extension and are considered modulo
$\Psi^{-\infty}$. The residues do not depend on the chosen
extensions or parametrices: the relevant interior symbols and
boundary jets are determined by the original metrics, and
smoothing remainders have zero residue.

For an operator $A$ satisfying the transmission condition,
write $\pi^+A=r^+Ae^+$, where $e^+$ extends a section by zero
and $r^+$ restricts it to the interior of $M$. We use the
noncommutative residue of \cite{FGLS1996} in the Boutet de Monvel
calculus. For $s\in\{1,2\}$, set
\[
P_s=D_{g_1}^2D_{g_2}^{-s},
\qquad
Q_s=D_{g_2}^{-n+s},
\qquad
\mathcal R_n^{(s)}(g_1\mid g_2)
=\Wresb[\pi^+P_s\circ\pi^+Q_s].
\]
The operators $P_s$ and $Q_s$ have orders $2-s$ and $-n+s$,
respectively. Differential operators and their integer-order
parametrices satisfy the transmission condition, and this
condition is preserved under composition. Hence it holds
for both factors.

We write
\[
P_s
=
D_{g_2}^{2-s}
+(D_{g_1}^2-D_{g_2}^2)D_{g_2}^{-s}
\quad\bmod\Psi^{-\infty}.
\]
For $s=1,2$, the first term is a differential operator and
produces no singular Green composition term. Thus the
singular Green contribution comes entirely from the second
term; in particular,
\[
\pi^+(D_{g_2}^{2-s})\circ\pi^+Q_s
=\pi^+(D_{g_2}^{2-s}Q_s).
\]
Indeed, a differential operator is local, so its action on an
extension by zero agrees in the interior with its action on
the original section. Consequently, modulo smoothing terms,
\begin{align*}
&\pi^+P_s\circ\pi^+Q_s-\pi^+(P_sQ_s)=
\pi^+\bigl((D_{g_1}^2-D_{g_2}^2)D_{g_2}^{-s}\bigr)
   \circ\pi^+Q_s
-\pi^+\bigl((D_{g_1}^2-D_{g_2}^2)D_{g_2}^{-n}\bigr).
\end{align*}

The collar assumption also gives
\[
\sigma_2(D_{g_1}^2-D_{g_2}^2)
=(h_1-h_2)|\xi'|_{g^{\partial M}}^2\Id_E.
\]
Thus the second-order normal derivatives cancel, and there
are no mixed second-order normal derivatives. The difference
$D_{g_1}^2-D_{g_2}^2$ contains at most one normal derivative.
After composition with $D_{g_2}^{-s}$, its normal symbol order
is therefore at most $1-s\le0$. Here normal order refers to
growth in $\xi_n$ with the tangential covariable fixed, rather
than to the total pseudodifferential order. The remaining
composition is consequently covered by the normal symbol
calculus in \cite{Boutet1971,Grubb1996}.

Write a covector as $\xi=\xi'+\xi_n dx_n$, where
$\xi'=\sum_{a<n}\xi_a dx^a$. We denote the Hardy projection
in the normal covariable by $\pi_{\xi_n}^+$. For the rational
symbols occurring below, this projection retains the principal
parts at poles in the upper half-plane and annihilates the
polynomial part. The partial fraction decompositions
\[
\frac1{1+\xi_n^2}
=\frac1{2\ii}
 \left(\frac1{\xi_n-\ii}-\frac1{\xi_n+\ii}\right),
\qquad
\frac{\xi_n}{1+\xi_n^2}
=\frac12
 \left(\frac1{\xi_n-\ii}+\frac1{\xi_n+\ii}\right)
\]
give
\begin{equation}
\label{eq:Hardy}
\pi_{\xi_n}^+\frac1{1+\xi_n^2}
=\frac1{2\ii(\xi_n-\ii)},
\qquad
\pi_{\xi_n}^+\frac{\xi_n}{1+\xi_n^2}
=\frac1{2(\xi_n-\ii)}.
\end{equation}
This projection on symbols is distinct from the truncation
of operators. In particular, the vanishing of the Hardy
projection of a polynomial symbol does not mean that its
operator truncation vanishes. The differential part separated
above contributes to the interior product, but not to the
singular Green term. We write
$\sigma_r^+(A)=\pi_{\xi_n}^+\sigma_r(A)$.

The composition formula and the noncommutative residue give
the following decomposition, in the convention of
\cite[(2.2.4)--(2.2.5)]{Wang2007}:
\begin{equation}
\label{eq:FGLS-split}
\mathcal R_n^{(s)}
=
\int_M\int_{|\xi|_{g_2}=1}
\trE\sigma_{-n}(P_sQ_s)
\,\dd S_{g_2}(\xi)\,\dd\vol_{g_2}
+
\int_{\partial M}\Phi_s(x')\,\dd\vol_{g^{\partial M}}(x').
\end{equation}
The local boundary coefficient is
\begin{align}
\Phi_s(x')={}&
\int_{|\xi'|_{g^{\partial M}}=1}\int_{\mathbb R}
\sum
\frac{(-\ii)^{|\alpha|+j+k+1}}
     {\alpha!(j+k+1)!}
\trE\Bigl[
\partial_{x_n}^{\,j}
\partial_{\xi'}^\alpha
\partial_{\xi_n}^{\,k}
\sigma_r^+(P_s)
\notag\\
&\hspace{30mm}\cdot
\partial_{x'}^\alpha
\partial_{\xi_n}^{\,j+1}
\partial_{x_n}^{\,k}
\sigma_\ell(Q_s)
\Bigr]_{x_n=0}
\,\dd\xi_n\,\dd S_{g^{\partial M}}(\xi').
\label{eq:FGLS}
\end{align}
Here $\alpha\in\mathbb N_0^{n-1}$, $j,k\in\mathbb N_0$, and
$r,\ell\in\mathbb Z$. The sum is taken over
\begin{equation}
\label{eq:defect}
r\le2-s,
\qquad
\ell\le-n+s,
\qquad
r+\ell-|\alpha|-j-k-1=-n.
\end{equation}
For Hodge operators with different factor orders, the same
local formula appears in
\cite[(3.18)]{LiuWuWang2022}.
The measure $\dd S_{g^{\partial M}}(\xi')$ is the tangential unit-sphere
measure. Thus $\Phi_s$ is the scalar coefficient after the
normal and tangential sphere integrations, while
$\dd\vol_{g^{\partial M}}$ is written separately in
\eqref{eq:FGLS-split}. All derivatives in \eqref{eq:FGLS}
are taken before setting $x_n=0$ and $|\xi'|_{g^{\partial M}}=1$.

The order condition can be rewritten as
\[
\bigl((2-s)-r\bigr)
+\bigl((-n+s)-\ell\bigr)
+|\alpha|+j+k=1.
\]
Each term on the left is a nonnegative integer. Hence exactly
one of them equals one, and the others vanish. This gives the
following five cases:
\[
\begin{array}{c|ccccc|c}
\text{case}
&r&\ell&j&k&|\alpha|
&\displaystyle
\frac{(-\ii)^{|\alpha|+j+k+1}}{\alpha!(j+k+1)!}
\\ \hline
(\mathrm a)(\mathrm I)
&2-s&-n+s&0&0&1&-1\\
(\mathrm a)(\mathrm {II})
&2-s&-n+s&1&0&0&-1/2\\
(\mathrm a)(\mathrm {III})
&2-s&-n+s&0&1&0&-1/2\\
(\mathrm b)
&1-s&-n+s&0&0&0&-\ii\\
(\mathrm c)
&2-s&-n+s-1&0&0&0&-\ii
\end{array}
\]
We denote their contributions, in this order, by
$\Phi_j^{(s)}$, $j=1,\ldots,5$. Each contribution includes
both integrations in \eqref{eq:FGLS}, so that
$\Phi_s=\sum_{j=1}^5\Phi_j^{(s)}$.

Finally, the parametrix identities give
$P_sQ_s=D_{g_1}^2D_{g_2}^{-n}$ modulo smoothing operators.
The local interior computation therefore yields
\begin{equation}
\label{eq:boundary-interior}
\int_M\int_{|\xi|_{g_2}=1}
\trE\sigma_{-n}(P_sQ_s)
\,\dd S_{g_2}(\xi)\,\dd\vol_{g_2}
=
2^{n-2}\nu_{n-1}
\int_M\mathscr L_n(g_1\mid g_2)\,\dd\vol_{g_2}.
\end{equation}
Here $\mathscr L_n$ is the local expression
\eqref{eq:local-density}. Its divergence term is retained,
since $M$ has boundary.

\subsection{Boundary connections, symbols, and exterior traces}

Fix $x_0\in\partial M$ and choose $g^{\partial M}$-normal coordinates
$x'$ centered there. Unless an argument is displayed, the
coefficients in this subsection are evaluated at $(x_0,0)$.
At this point the two covector Clifford actions coincide;
we denote their common value by $c$. Thus
$c(\xi')=\sum_{a<n}\xi_a c(dx^a)$ and
$c(\xi)=c(\xi')+\xi_n c(dx_n)$.
Away from the boundary, the Clifford actions retain their
metric subscripts.

\begin{lemma}[Boundary connection and first-order symbols]
\label{lem:warped-symbols}
For $r=1,2$, the possibly nonzero Christoffel coefficients
at $(x_0,0)$ are
\begin{equation}
\label{eq:boundary-Christoffel}
\Gamma(g_r)^n{}_{ab}
=\frac{h_r'(0)}2\delta_{ab},
\qquad
\Gamma(g_r)^a{}_{bn}
=\Gamma(g_r)^a{}_{nb}
=-\frac{h_r'(0)}2\delta^a_b.
\end{equation}
In the coordinate exterior frame, write
$\sigma_1(D_{g_r}^2)=\ii\sum_iG_r^i\xi_i$. Then
\begin{align}
G_r^a
&=
h_r'(0)
\bigl(
\eps(dx^a)\iotaop(\partial_n)
-\eps(dx_n)\iotaop(\partial_a)
\bigr),
\notag\\
G_r^n
&=
h_r'(0)
\left(
\frac{n-1}{2}\Id_E
-\sum_{a<n}\eps(dx^a)\iotaop(\partial_a)
\right).
\label{eq:warped-G}
\end{align}
Moreover,
\begin{align}
\sigma_0(D_{g_r})
&=
h_r'(0)
\left(
\frac{n-1}{2}\Id_E
-\sum_{a<n}\eps(dx^a)\iotaop(\partial_a)
\right)\iotaop(\partial_n),
\notag\\
\partial_{x_n}c_{g_r}(\xi)
&=
-h_r'(0)\sum_{a<n}\xi_a\iotaop(\partial_a).
\label{eq:boundary-Dirac}
\end{align}
The last derivative is taken with the coordinate covector
components $\xi_i$ held fixed.
\end{lemma}

\begin{proof}
At the chosen point,
\[
g_{r,ab}=\delta_{ab},
\qquad
g_{r,an}=0,
\qquad
g_{r,nn}=1,
\qquad
\partial_n g_{r,ab}=-h_r'(0)\delta_{ab},
\qquad
\partial_n g_r^{ab}=h_r'(0)\delta^{ab}.
\]
All tangential first derivatives of the metric components
vanish, since $h_r$ depends only on $x_n$ and $x'$ is normal
for $g^{\partial M}$. The Christoffel formula consequently gives
\begin{align*}
\Gamma(g_r)^n{}_{ab}
&=
\frac12
\bigl(
\partial_a g_{r,bn}
+\partial_b g_{r,an}
-\partial_n g_{r,ab}
\bigr)
=\frac{h_r'(0)}2\delta_{ab},
\\
\Gamma(g_r)^a{}_{bn}
&=
\frac12\delta^{ac}
\bigl(
\partial_b g_{r,nc}
+\partial_n g_{r,bc}
-\partial_c g_{r,bn}
\bigr)
=-\frac{h_r'(0)}2\delta^a_b.
\end{align*}
Symmetry in the lower indices gives
$\Gamma(g_r)^a{}_{nb}$, and the remaining coefficients vanish.

Write
$\nabla_{\partial_i}^{\Lambda,g_r}
=\partial_i+\omega_i^{(r)}$
in the coordinate exterior frame. By
Lemma~\ref{lem:exterior-connection},
$\omega_i^{(r)}
=-\Gamma(g_r)^a{}_{ib}\eps(dx^b)\iotaop(\partial_a)$.
Substituting \eqref{eq:boundary-Christoffel}, we find
\begin{align*}
\omega_a^{(r)}
&=
-\frac{h_r'(0)}2\eps(dx^a)\iotaop(\partial_n)
+\frac{h_r'(0)}2\eps(dx_n)\iotaop(\partial_a),
\\
\omega_n^{(r)}
&=
\frac{h_r'(0)}2
\sum_{a<n}\eps(dx^a)\iotaop(\partial_a).
\end{align*}
Expansion of the Weitzenb\"ock formula at
$g_r^{ij}=\delta^{ij}$ gives
\[
G_r^k
=-2\omega_k^{(r)}
+\sum_i\Gamma(g_r)^k{}_{ii}\Id_E.
\]
For $k=a<n$, the Christoffel contraction is zero. For $k=n$,
it equals $(n-1)h_r'(0)/2$. These two cases give
\eqref{eq:warped-G}.

For the first-order Hodge operator, the coordinate formulas are
\[
d=\sum_i\eps(dx^i)\partial_i,
\qquad
\delta_{g_r}
=-\sum_{i,j}g_r^{ij}
 \iotaop(\partial_i)
 \bigl(\partial_j+\omega_j^{(r)}\bigr).
\]
Therefore the zero-order coefficient of $D_{g_r}=d+\delta_{g_r}$
at the boundary point is
\begin{align*}
\sigma_0(D_{g_r})
&=
-\sum_i\iotaop(\partial_i)\omega_i^{(r)}
\\
&=
\frac{h_r'(0)}2\sum_{a<n}
\Bigl[
\iotaop(\partial_a)\eps(dx^a)\iotaop(\partial_n)
-\iotaop(\partial_a)\eps(dx_n)\iotaop(\partial_a)
-\iotaop(\partial_n)\eps(dx^a)\iotaop(\partial_a)
\Bigr].
\end{align*}
For each $a<n$, the middle term vanishes:
$\iotaop(\partial_a)\eps(dx_n)\iotaop(\partial_a)
=-\eps(dx_n)\iotaop(\partial_a)^2=0$.
The other two terms satisfy
\begin{align*}
\iotaop(\partial_a)\eps(dx^a)\iotaop(\partial_n)
&=
\bigl(
\Id_E-\eps(dx^a)\iotaop(\partial_a)
\bigr)\iotaop(\partial_n),
\\
\iotaop(\partial_n)\eps(dx^a)\iotaop(\partial_a)
&=
\eps(dx^a)\iotaop(\partial_a)\iotaop(\partial_n).
\end{align*}
Their difference is
$\bigl(\Id_E-2\eps(dx^a)\iotaop(\partial_a)\bigr)
\iotaop(\partial_n)$.
Summing over $a<n$ proves the first identity in
\eqref{eq:boundary-Dirac}.

In the coordinate exterior frame, exterior multiplication by
$dx^i$ and contraction by $\partial_i$ have constant matrices.
The covector Clifford action is
\[
c_{g_r}(\xi)
=\eps(\xi)-\sum_{i,j}g_r^{ij}\xi_i\iotaop(\partial_j).
\]
Thus, with $\xi_i$ fixed,
\[
\partial_n c_{g_r}(\xi)
=-\sum_{i,j}(\partial_n g_r^{ij})
  \xi_i\iotaop(\partial_j)
=-h_r'(0)\sum_{a<n}\xi_a\iotaop(\partial_a).
\]
This proves the second identity in \eqref{eq:boundary-Dirac}.
\end{proof}

\begin{lemma}[Boundary traces]
\label{lem:boundary-traces}
At $(x_0,0)$, for $r=1,2$ and $|\xi'|_{g^{\partial M}}=1$,
\begin{align}
\trE G_r^i&=0,
\notag\\
\trE\bigl[G_r^n c(\xi')c(dx_n)\bigr]&=0,
\notag\\
\sum_{a<n}\xi_a
\trE\bigl[G_r^a c(\xi')c(dx_n)\bigr]
&=2^{n-1}h_r'(0).
\label{eq:boundary-traces}
\end{align}
\end{lemma}

\begin{proof}
For $a<n$, \eqref{eq:lambda2} gives
$\trE(\eps(dx^a)\iotaop(\partial_n))
=\trE(\eps(dx_n)\iotaop(\partial_a))=0$.
Hence $\trE G_r^a=0$. For the normal coefficient,
\[
\trE G_r^n
=
h_r'(0)
\left(
\frac{n-1}{2}2^n-(n-1)2^{n-1}
\right)
=0.
\]

The endomorphisms $\Id_E$ and
$\eps(dx^a)\iotaop(\partial_a)$ are diagonal in the coordinate
exterior basis. Every nonzero term of $c(\xi')c(dx_n)$ changes
whether $dx_n$ occurs in a basis form. Multiplication by
$G_r^n$ does not change this property. Thus
$G_r^n c(\xi')c(dx_n)$ has zero diagonal entries, proving
the second identity.

For the last identity, we first compute, for $a,b<n$,
\[
\trE\Bigl[
\bigl(
\eps(dx^a)\iotaop(\partial_n)
-\eps(dx_n)\iotaop(\partial_a)
\bigr)c(dx^b)c(dx_n)
\Bigr]
=2^{n-1}\delta_{ab}.
\]
When $a\ne b$, every nonzero image changes the tangential
index set, so the trace vanishes. When $a=b$, expansion
using \eqref{eq:CAR} reduces the operator inside the trace to
\[
\eps(dx^a)\iotaop(\partial_a)
+\eps(dx_n)\iotaop(\partial_n)
-2\eps(dx^a)\iotaop(\partial_a)
   \eps(dx_n)\iotaop(\partial_n).
\]
Its trace, by \eqref{eq:lambda2} and \eqref{eq:lambda4}, is
$2^{n-1}+2^{n-1}-2\cdot2^{n-2}=2^{n-1}$.
Consequently,
\begin{align*}
\sum_{a<n}\xi_a
\trE\bigl[G_r^a c(\xi')c(dx_n)\bigr]
&=
2^{n-1}h_r'(0)
\sum_{a,b<n}\xi_a\xi_b\delta_{ab}
\\
&=
2^{n-1}h_r'(0)|\xi'|_{g^{\partial M}}^2
=2^{n-1}h_r'(0).
\end{align*}
\end{proof}

For later use in the odd splitting, the preceding identities
also give, on $|\xi'|_{g^{\partial M}}=1$,
\begin{align}
&\trE\Biggl[
\Biggl\{
\sum_{a<n}\xi_a
\bigl(
\eps(dx^a)\iotaop(\partial_n)
-\eps(dx_n)\iotaop(\partial_a)
\bigr)
\notag\\
&\hspace{11mm}
+\ii\left(
\frac{n-1}{2}\Id_E
-\sum_{a<n}\eps(dx^a)\iotaop(\partial_a)
\right)
\Biggr\}
(c(\xi')+\ii c(dx_n))c(\xi)
\Biggr]
\notag\\
&
=2^{n-1}(\xi_n-\ii).
\label{eq:odd-contraction}
\end{align}
Indeed, $c(\xi')^2=c(dx_n)^2=-\Id_E$ and
$c(\xi')c(dx_n)+c(dx_n)c(\xi')=0$ imply
\[
(c(\xi')+\ii c(dx_n))c(\xi)
=-(1+\ii\xi_n)\Id_E
 +(\xi_n-\ii)c(\xi')c(dx_n).
\]
The expression in braces in \eqref{eq:odd-contraction} has
zero trace. Its tangential part, multiplied by
$c(\xi')c(dx_n)$, has trace $2^{n-1}$ by the preceding
calculation, whereas its normal part has trace zero.
Substituting these identities proves
\eqref{eq:odd-contraction}.

\subsection{The normal contour integrals}

After taking the exterior traces, the remaining normal
integrals are scalar rational integrals. We evaluate them by
closing the contour in the upper half-plane. The integrands
below have no real poles and decay sufficiently rapidly for
the semicircle integrals to vanish.

\begin{lemma}
\label{lem:contours}
For integers $p\ge1$ and $m\ge2$,
\begin{align}
J_p
&:=
\int_{\mathbb R}
\frac1{2\ii(\xi_n-\ii)}
\frac{d^2}{d\xi_n^2}(1+\xi_n^2)^{-p}\,\dd\xi_n
=
-\frac{2\pi(2p+1)!}
{2^{2p+2}(p-1)!(p+2)!},
\label{eq:Jp}
\\
L_m
&:=
\int_{\mathbb R}
\frac{(1+\xi_n^2)^{-m}}{(\xi_n-\ii)^2}\,\dd\xi_n
=
-\frac{\pi(2m)!}
{2^{2m}(m-1)!(m+1)!},
\label{eq:Lm}
\\
M_m
&:=
\int_{\mathbb R}
\frac{(1+\xi_n^2)^{-m}}{\xi_n-\ii}\,\dd\xi_n
=
\frac{2\pi\ii(2m-1)!}
{2^{2m}(m-1)!m!}.
\label{eq:Mm}
\end{align}
\end{lemma}

\begin{proof}
For positive integers $a,b$, the rational function
$(\xi_n-\ii)^{-a}(\xi_n+\ii)^{-b}$ has a pole of order $a$
at $\xi_n=\ii$. Its residue is
\begin{align*}
\operatorname*{Res}_{\xi_n=\ii}
\frac1{(\xi_n-\ii)^a(\xi_n+\ii)^b}
&=
\frac1{(a-1)!}
\left[
\frac{d^{a-1}}{d\xi_n^{a-1}}
(\xi_n+\ii)^{-b}
\right]_{\xi_n=\ii}
\\
&=
\frac{(-1)^{a-1}(a+b-2)!}
{(a-1)!(b-1)!(2\ii)^{a+b-1}}.
\end{align*}
There are no other poles in the upper half-plane. Since
$a+b\ge2$, the semicircle integral tends to zero, and the
real integral equals $2\pi\ii$ times this residue.

For $L_m$, the factorization
$1+\xi_n^2=(\xi_n-\ii)(\xi_n+\ii)$ gives
$(a,b)=(m+2,m)$. Hence
\[
L_m
=
2\pi\ii\,
\frac{(-1)^{m+1}(2m)!}
{(m+1)!(m-1)!(2\ii)^{2m+1}}
=
-\frac{\pi(2m)!}
{2^{2m}(m-1)!(m+1)!}.
\]
For $M_m$, take $(a,b)=(m+1,m)$:
\[
M_m
=
2\pi\ii\,
\frac{(-1)^m(2m-1)!}
{m!(m-1)!(2\ii)^{2m}}
=
\frac{2\pi\ii(2m-1)!}
{2^{2m}m!(m-1)!}.
\]

For $J_p$, integration by parts twice gives
\begin{align*}
J_p
&=
-\frac1{2\ii}\int_{\mathbb R}
\frac{d}{d\xi_n}(\xi_n-\ii)^{-1}
\frac{d}{d\xi_n}(1+\xi_n^2)^{-p}\,\dd\xi_n
\\
&=
\frac1{2\ii}\int_{\mathbb R}
\frac{d^2}{d\xi_n^2}(\xi_n-\ii)^{-1}
(1+\xi_n^2)^{-p}\,\dd\xi_n
\\
&=
\frac1{\ii}\int_{\mathbb R}
\frac{\dd\xi_n}
{(\xi_n-\ii)^{p+3}(\xi_n+\ii)^p}.
\end{align*}
The boundary terms vanish because the rational functions
and their derivatives tend to zero at infinity. Applying
the residue formula with $(a,b)=(p+3,p)$ yields
\[
J_p
=
2\pi\,
\frac{(-1)^{p+2}(2p+1)!}
{(p+2)!(p-1)!(2\ii)^{2p+2}}
=
-\frac{2\pi(2p+1)!}
{2^{2p+2}(p-1)!(p+2)!}.
\]
This proves all three identities.
\end{proof}

\section{Boundary contributions and the bimetric KKW formulas}
\label{sec:boundary-calculation}

Throughout this section, $n=2m\ge4$ and the metrics satisfy
\eqref{eq:collar}. We apply \eqref{eq:FGLS} to the two
factorizations, using the five cases described in
Section~\ref{sec:boundary-framework}. The calculation consists
of composing the symbols, taking the normal Hardy projection,
evaluating the exterior traces, and performing the normal
and tangential sphere integrations.

Fix $x_0\in\partial M$ and choose $g^{\partial M}$-normal coordinates
$x'$ centered there. We use the common coordinate exterior
frame of Section~\ref{sec:boundary-framework}. Unless otherwise
stated, symbol identities below are evaluated at $(x_0,0)$,
after the indicated base derivatives have been taken.
Restriction to $|\xi'|_{g^{\partial M}}=1$ is made only after any
required tangential covariable derivatives.

\subsection{The even exponent splitting}
\label{sec:even-boundary}

Let $P_2=D_{g_1}^2D_{g_2}^{-2}$ and
$Q_2=D_{g_2}^{-n+2}$. In the collar,
\[
|\xi|_{g_r}^2
=h_r(x_n)|\xi'|_{g^{\partial M}}^2+\xi_n^2,
\qquad
|\xi'|_{g^{\partial M}}^2
=(g^{\partial M})^{ab}(x')\xi_a\xi_b,
\]
where $((g^{\partial M})^{ab})$ is the inverse matrix of
$((g^{\partial M})_{ab})$. At $(x_0,0)$ the two quadratic symbols
coincide:
$|\xi|_{g_1}^2=|\xi|_{g_2}^2=|\xi'|_{g^{\partial M}}^2+\xi_n^2$. Their first derivatives
at this point satisfy
\[
\partial_{x_a}|\xi|_{g_r}^2=0\quad(a<n),
\qquad
\partial_{x_n}|\xi|_{g_r}^2
=h_r'(0)|\xi'|_{g^{\partial M}}^2,
\qquad
\partial_{\xi_n}|\xi|_{g_r}^2=2\xi_n.
\]

\paragraph{The symbols of the two factors.}

Write
$q_{-2}=\sigma_{-2}(D_{g_2}^{-2})$ and
$q_{-3}=\sigma_{-3}(D_{g_2}^{-2})$
for the first two homogeneous symbol components.
Using $D_{x_j}=-\ii\partial_{x_j}$, the terms of degrees
zero and minus one in the parametrix identity
$\sigma(D_{g_2}^2)\#\sigma(D_{g_2}^{-2})=\Id_E$ give
\begin{align*}
|\xi|_{g_2}^2q_{-2}&=\Id_E,
\\
|\xi|_{g_2}^2q_{-3}
+\ii\sum_iG_2^i\xi_iq_{-2}
+\sum_j
 \partial_{\xi_j}|\xi|_{g_2}^2D_{x_j}q_{-2}
&=0.
\end{align*}
where $\#$ denotes the symbol product in
\eqref{eq:symbol-product}.
Before evaluation at the boundary,
$q_{-2}=|\xi|_{g_2}^{-2}\Id_E$. Hence
\[
D_{x_n}q_{-2}
=-\ii\partial_{x_n}(|\xi|_{g_2}^{-2})\Id_E
=\ii h_2'(0)|\xi'|_{g^{\partial M}}^2|\xi|_{g_2}^{-4}\Id_E.
\]
Its tangential base derivatives vanish at the chosen point.
It follows that
\begin{align}
q_{-3}
&=
-|\xi|_{g_2}^{-2}
\left[
\ii\sum_{i=1}^n G_2^i\xi_i\,|\xi|_{g_2}^{-2}
+
2\ii h_2'(0)\xi_n
|\xi'|_{g^{\partial M}}^2
|\xi|_{g_2}^{-4}\Id_E
\right]
\notag\\
&=
-\ii\sum_{i=1}^n G_2^i\xi_i\,|\xi|_{g_2}^{-4}
-
2\ii h_2'(0)\xi_n
|\xi'|_{g^{\partial M}}^2
|\xi|_{g_2}^{-6}\Id_E.
\label{eq:inverse-square-subleading}
\end{align}
By \eqref{eq:warped-G}, the matrix coefficient is
\begin{align*}
\sum_iG_r^i\xi_i
=h_r'(0)\Biggl[
&\sum_{a<n}\xi_a
 \bigl(
 \eps(dx^a)\iotaop(\partial_n)
 -\eps(dx_n)\iotaop(\partial_a)
 \bigr)
+\xi_n
 \left(
 \frac{n-1}{2}\Id_E
 -\sum_{a<n}\eps(dx^a)\iotaop(\partial_a)
 \right)
\Biggr].
\end{align*}

The leading symbol of $P_2$ is
$\sigma_0(P_2)=|\xi|_{g_1}^2|\xi|_{g_2}^{-2}\Id_E$.
At the boundary point, the next symbol has three
contributions:
\begin{align*}
\sigma_{-1}(P_2)
&=
|\xi|_{g_1}^2q_{-3}
+\ii\sum_iG_1^i\xi_i|\xi|_{g_2}^{-2}
+2\xi_nD_{x_n}q_{-2}
\\
&=
-\ii\sum_iG_2^i\xi_i|\xi|_{g_2}^{-2}
-2\ii h_2'(0)\xi_n|\xi'|_{g^{\partial M}}^2|\xi|_{g_2}^{-4}\Id_E
\\
&\quad+\ii\sum_iG_1^i\xi_i|\xi|_{g_2}^{-2}
+2\ii h_2'(0)\xi_n|\xi'|_{g^{\partial M}}^2|\xi|_{g_2}^{-4}\Id_E
\\
&=
\ii\sum_i(G_1^i-G_2^i)\xi_i|\xi|_{g_2}^{-2}.
\end{align*}
The scalar derivative terms cancel. Substituting
\eqref{eq:warped-G}, we obtain
\begin{align*}
\sigma_{-1}(P_2)
=
\frac{\ii(h_1'(0)-h_2'(0))}{|\xi|_{g_2}^2}
\Biggl[
&\sum_{a<n}\xi_a
 \bigl(
 \eps(dx^a)\iotaop(\partial_n)
 -\eps(dx_n)\iotaop(\partial_a)
 \bigr)
\\
&+\xi_n
 \left(
 \frac{n-1}{2}\Id_E
 -\sum_{a<n}\eps(dx^a)\iotaop(\partial_a)
 \right)
\Biggr].
\end{align*}

Although $\sigma_0(P_2)=\Id_E$ at the boundary, its normal
derivative must be computed from the collar expression:
\begin{align*}
\partial_{x_n}\sigma_0(P_2)
&=
\left[
\frac{\partial_{x_n}|\xi|_{g_1}^2}{|\xi|_{g_2}^2}
-\frac{|\xi|_{g_1}^2\partial_{x_n}|\xi|_{g_2}^2}
      {|\xi|_{g_2}^4}
\right]\Id_E
\\
&=
\frac{(h_1'(0)-h_2'(0))|\xi'|_{g^{\partial M}}^2}
     {|\xi'|_{g^{\partial M}}^2+\xi_n^2}\Id_E.
\end{align*}
On $|\xi'|_{g^{\partial M}}=1$, the projection formulas
\eqref{eq:Hardy} therefore give
\begin{align}
\pi_{\xi_n}^+\sigma_0(P_2)&=0,
\notag\\
\partial_{x_n}\pi_{\xi_n}^+\sigma_0(P_2)
&=
\frac{h_1'(0)-h_2'(0)}
     {2\ii(\xi_n-\ii)}\Id_E,
\notag\\
\pi_{\xi_n}^+\sigma_{-1}(P_2)
&=
\frac{h_1'(0)-h_2'(0)}{2(\xi_n-\ii)}
\Biggl[
\sum_{a<n}\xi_a
 \bigl(
 \eps(dx^a)\iotaop(\partial_n)
 -\eps(dx_n)\iotaop(\partial_a)
 \bigr)
\notag\\
&\hspace{24mm}
+\ii\left(
\frac{n-1}{2}\Id_E
-\sum_{a<n}\eps(dx^a)\iotaop(\partial_a)
\right)
\Biggr].
\label{eq:even-symbols}
\end{align}
The Hardy projection acts in $\xi_n$ and commutes with
the base derivatives used here. In particular, the second
line is obtained by differentiating the symbol before
restricting it to the boundary.

Since $Q_2=(D_{g_2}^2)^{-(m-1)}$ modulo smoothing operators,
its leading symbol is
$\sigma_{-n+2}(Q_2)=|\xi|_{g_2}^{-n+2}\Id_E$.
On the unit tangential sphere, the relevant normal
covariable derivatives are
\begin{align*}
\partial_{\xi_n}(1+\xi_n^2)^{-m+1}
&=
-2(m-1)\xi_n(1+\xi_n^2)^{-m},
\\
\partial_{\xi_n}^2(1+\xi_n^2)^{-m+1}
&=
-2(m-1)(1+\xi_n^2)^{-m}
+4m(m-1)\xi_n^2(1+\xi_n^2)^{-m-1}.
\end{align*}

\paragraph{Case (a)(I).}

Here $r=0$, $\ell=-n+2$, $j=k=0$, and $|\alpha|=1$.
For each such multi-index, $\alpha!=1$, so the coefficient
in \eqref{eq:FGLS} is $(-\ii)^2=-1$. Thus
\begin{align*}
\Phi_1^{(2)}
=
-\int_{|\xi'|_{g^{\partial M}}=1}\int_{\mathbb R}
\sum_{a<n}
\trE\left[
\partial_{\xi_a}\pi_{\xi_n}^+\sigma_0(P_2)\,
\partial_{x_a}\partial_{\xi_n}
\sigma_{-n+2}(Q_2)
\right]
\,\dd\xi_n\,\dd S_{g^{\partial M}}(\xi').
\end{align*}
Before evaluating at $(x_0,0)$,
\[
\partial_{x_a}\sigma_{-n+2}(Q_2)
=
-(m-1)|\xi|_{g_2}^{-2m}
 \partial_{x_a}|\xi|_{g_2}^2\Id_E.
\]
At the chosen point,
$\partial_{x_a}|\xi|_{g_2}^2
=h_2(\partial_{x_a}(g^{\partial M})^{bc})\xi_b\xi_c=0$.
This holds for every covector, and hence its
$\xi_n$-derivative also vanishes. Substitution into the
integral gives $\Phi_1^{(2)}=0$.

\paragraph{Case (a)(II).}

Now $r=0$, $\ell=-n+2$, $j=1$, and $k=|\alpha|=0$.
The coefficient in \eqref{eq:FGLS} is
$(-\ii)^2/2!=-1/2$. By \eqref{eq:even-symbols},
\begin{align}
\Phi_2^{(2)}
&=
-\frac12
\int_{|\xi'|_{g^{\partial M}}=1}\int_{\mathbb R}
\trE\left[
\partial_{x_n}\pi_{\xi_n}^+\sigma_0(P_2)\,
\partial_{\xi_n}^2\sigma_{-n+2}(Q_2)
\right]
\,\dd\xi_n\,\dd S_{g^{\partial M}}(\xi')
\notag\\
&=
-\frac12
\int_{|\xi'|_{g^{\partial M}}=1}\int_{\mathbb R}
\frac{h_1'(0)-h_2'(0)}{2\ii(\xi_n-\ii)}
\frac{d^2}{d\xi_n^2}(1+\xi_n^2)^{-m+1}
\trE(\Id_E)
\,\dd\xi_n\,\dd S_{g^{\partial M}}(\xi')
\notag\\
&=
-2^{n-1}(h_1'(0)-h_2'(0))\nu_{n-2}J_{m-1}
\notag\\
&=
\pi\nu_{n-2}(h_1'(0)-h_2'(0))
\frac{(2m-1)!}{(m-2)!(m+1)!}.
\label{eq:even-b}
\end{align}
Here $\trE(\Id_E)=2^n$, and the traced integrand is independent
of the direction of $\xi'$. The last equality follows from
\eqref{eq:Jp} with $p=m-1$ and $n=2m$.

\paragraph{Case (a)(III).}

Take $r=0$, $\ell=-n+2$, $k=1$, and $j=|\alpha|=0$.
The coefficient is again $-1/2$, so
\[
\Phi_3^{(2)}
=
-\frac12
\int_{|\xi'|_{g^{\partial M}}=1}\int_{\mathbb R}
\trE\left[
\partial_{\xi_n}\pi_{\xi_n}^+\sigma_0(P_2)\,
\partial_{\xi_n}\partial_{x_n}
\sigma_{-n+2}(Q_2)
\right]
\,\dd\xi_n\,\dd S_{g^{\partial M}}(\xi').
\]
The right factor is determined by
\begin{align*}
\partial_{x_n}\sigma_{-n+2}(Q_2)
&=
-(m-1)h_2'(0)(1+\xi_n^2)^{-m}\Id_E,
\\
\partial_{\xi_n}\partial_{x_n}\sigma_{-n+2}(Q_2)
&=
2m(m-1)h_2'(0)\xi_n
(1+\xi_n^2)^{-m-1}\Id_E.
\end{align*}
The left factor is zero by \eqref{eq:even-symbols}.
Therefore $\Phi_3^{(2)}=0$.

\paragraph{Case (b).}

Here $r=-1$, $\ell=-n+2$, and $j=k=|\alpha|=0$.
The coefficient in \eqref{eq:FGLS} is $-\ii$. Thus
\begin{align*}
\Phi_4^{(2)}
&=
-\ii
\int_{|\xi'|_{g^{\partial M}}=1}\int_{\mathbb R}
\trE\left[
\pi_{\xi_n}^+\sigma_{-1}(P_2)\,
\partial_{\xi_n}\sigma_{-n+2}(Q_2)
\right]
\,\dd\xi_n\,\dd S_{g^{\partial M}}(\xi')
\\
&=
2\ii(m-1)
\int_{|\xi'|_{g^{\partial M}}=1}\int_{\mathbb R}
\xi_n(1+\xi_n^2)^{-m}
\trE\bigl[\pi_{\xi_n}^+\sigma_{-1}(P_2)\bigr]
\,\dd\xi_n\,\dd S_{g^{\partial M}}(\xi').
\end{align*}
By Lemma~\ref{lem:boundary-traces},
\begin{align*}
\trE\bigl[\pi_{\xi_n}^+\sigma_{-1}(P_2)\bigr]
&=
\pi_{\xi_n}^+
\left[
\frac{\ii}{|\xi|_{g_2}^2}
\sum_i\xi_i\bigl(\trE G_1^i-\trE G_2^i\bigr)
\right]
=0.
\end{align*}
Hence $\Phi_4^{(2)}=0$.

\paragraph{Case (c).}

The remaining choice is $r=0$, $\ell=-n+1$, and
$j=k=|\alpha|=0$, with coefficient $-\ii$. Consequently,
\begin{align*}
\Phi_5^{(2)}
&=
-\ii
\int_{|\xi'|_{g^{\partial M}}=1}\int_{\mathbb R}
\trE\left[
\pi_{\xi_n}^+\sigma_0(P_2)\,
\partial_{\xi_n}\sigma_{-n+1}(Q_2)
\right]
\,\dd\xi_n\,\dd S_{g^{\partial M}}(\xi')
\\
&=0
\end{align*}
by $\pi_{\xi_n}^+\sigma_0(P_2)=0$.
The subleading symbol of $Q_2$ therefore does not enter
the calculation.

\begin{proposition}[Even boundary contribution]
\label{prop:even-boundary}
Under \eqref{eq:collar},
\[
\int_{\partial M}\Phi_2\,\dd\vol_{g^{\partial M}}
=
\pi\nu_{n-2}
\frac{(2m-1)!}{(m-2)!(m+1)!}
\int_{\partial M}
(h_1'(0)-h_2'(0))\,\dd\vol_{g^{\partial M}}.
\]
\end{proposition}

\begin{proof}
The five cases give
\[
\Phi_2
=
\Phi_1^{(2)}+\Phi_2^{(2)}+\Phi_3^{(2)}
+\Phi_4^{(2)}+\Phi_5^{(2)}
=
\Phi_2^{(2)}.
\]
Substituting \eqref{eq:even-b} and integrating over
$\partial M$ proves the assertion.
\end{proof}

\subsection{The odd exponent splitting}
\label{sec:odd-boundary}

Let $P_1=D_{g_1}^2D_{g_2}^{-1}$ and
$Q_1=D_{g_2}^{-n+1}$. The interior parametrix identity
$P_1=P_2D_{g_2}$ modulo smoothing operators allows us to
reuse the symbols of $P_2$. The symbol compositions below
are performed before truncating the operators.

\paragraph{The symbols of the two factors.}

The leading symbol is
\[
\sigma_1(P_1)
=
\frac{|\xi|_{g_1}^2}{|\xi|_{g_2}^2}\,
\ii c_{g_2}(\xi).
\]
Differentiating the quotient and the Clifford factor
separately, then evaluating at the boundary, gives
\begin{align*}
\partial_{x_n}\sigma_1(P_1)
&=
\frac{\ii(h_1'(0)-h_2'(0))|\xi'|_{g^{\partial M}}^2}
     {|\xi'|_{g^{\partial M}}^2+\xi_n^2}c(\xi)
-\ii h_2'(0)
 \sum_{a<n}\xi_a\iotaop(\partial_a).
\end{align*}
At the boundary, $\sigma_1(P_1)=\ii c(\xi)$ is polynomial
in $\xi_n$, so its Hardy projection vanishes. The second
term in its normal derivative is also polynomial.
Thus, on $|\xi'|_{g^{\partial M}}=1$,
\begin{align}
\pi_{\xi_n}^+\sigma_1(P_1)&=0,
\notag\\
\partial_{x_n}\pi_{\xi_n}^+\sigma_1(P_1)
&=
\ii(h_1'(0)-h_2'(0))
\left[
\frac{c(\xi')}{2\ii(\xi_n-\ii)}
+\frac{c(dx_n)}{2(\xi_n-\ii)}
\right]
\notag\\
&=
\frac{h_1'(0)-h_2'(0)}{2(\xi_n-\ii)}
\bigl(c(\xi')+\ii c(dx_n)\bigr).
\label{eq:odd-normal}
\end{align}

The next symbol degree in $P_1=P_2D_{g_2}$ is
\begin{align*}
\sigma_0(P_1)
={}&
\sigma_0(P_2)\sigma_0(D_{g_2})
+\sigma_{-1}(P_2)\,\ii c_{g_2}(\xi)
+
\sum_j\partial_{\xi_j}\sigma_0(P_2)\,
D_{x_j}\bigl(\ii c_{g_2}(\xi)\bigr).
\end{align*}
At the boundary, $\sigma_0(P_2)=\Id_E$ for every nonzero
covector, and hence all its covariable derivatives vanish.
It follows that
\begin{align*}
\sigma_0(P_1)
={}&
\sigma_0(D_{g_2})
-\frac{h_1'(0)-h_2'(0)}{|\xi|_{g_2}^2}
\Biggl[
\sum_{a<n}\xi_a
 \bigl(
 \eps(dx^a)\iotaop(\partial_n)
 -\eps(dx_n)\iotaop(\partial_a)
 \bigr)
\\
&\hspace{27mm}
+\xi_n
 \left(
 \frac{n-1}{2}\Id_E
 -\sum_{a<n}\eps(dx^a)\iotaop(\partial_a)
 \right)
\Biggr]c(\xi).
\end{align*}

Restrict now to $|\xi'|_{g^{\partial M}}=1$. The first term is
independent of $\xi_n$. The second is a rational matrix
function with denominator
$(\xi_n-\ii)(\xi_n+\ii)$ and numerator of degree at most two
in $\xi_n$. Its polynomial part is removed by the Hardy
projection, while its principal part at $\xi_n=\ii$ is
obtained by evaluating the numerator there. Therefore
\begin{align}
\pi_{\xi_n}^+\sigma_0(P_1)
={}&
-\frac{h_1'(0)-h_2'(0)}{2\ii(\xi_n-\ii)}
\Biggl[
\sum_{a<n}\xi_a
 \bigl(
 \eps(dx^a)\iotaop(\partial_n)
 -\eps(dx_n)\iotaop(\partial_a)
 \bigr)
\notag\\
&\hspace{18mm}
+\ii\left(
\frac{n-1}{2}\Id_E
-\sum_{a<n}\eps(dx^a)\iotaop(\partial_a)
\right)
\Biggr]
\bigl(c(\xi')+\ii c(dx_n)\bigr).
\label{eq:odd-zero}
\end{align}
The matrix numerator in this expression is independent
of $\xi_n$.

For the right factor, the identity
$Q_1=D_{g_2}(D_{g_2}^2)^{-m}$ modulo smoothing operators
gives the leading symbol
$\ii c_{g_2}(\xi)|\xi|_{g_2}^{-2m}$.
At the boundary, on $|\xi'|_{g^{\partial M}}=1$, this becomes
\begin{equation}
\label{eq:odd-Q}
\sigma_{-n+1}(Q_1)
=
\ii c(\xi)(1+\xi_n^2)^{-m}.
\end{equation}
Its first two normal covariable derivatives are
\begin{align*}
\partial_{\xi_n}\sigma_{-n+1}(Q_1)
&=
\ii c(dx_n)(1+\xi_n^2)^{-m}
-2m\ii\xi_nc(\xi)(1+\xi_n^2)^{-m-1},
\\
\partial_{\xi_n}^2\sigma_{-n+1}(Q_1)
&=
-4m\ii\xi_nc(dx_n)(1+\xi_n^2)^{-m-1}
\\
&\quad
+\ii c(\xi)
\left[
-2m(1+\xi_n^2)^{-m-1}
+4m(m+1)\xi_n^2(1+\xi_n^2)^{-m-2}
\right].
\end{align*}

\paragraph{Case (a)(I).}

Here $r=1$, $\ell=-n+1$, $j=k=0$, and $|\alpha|=1$.
The coefficient in \eqref{eq:FGLS} is $-1$, so
\[
\Phi_1^{(1)}
=
-\int_{|\xi'|_{g^{\partial M}}=1}\int_{\mathbb R}
\sum_{a<n}
\trE\left[
\partial_{\xi_a}\pi_{\xi_n}^+\sigma_1(P_1)\,
\partial_{x_a}\partial_{\xi_n}
\sigma_{-n+1}(Q_1)
\right]
\,\dd\xi_n\,\dd S_{g^{\partial M}}(\xi').
\]
Before evaluation at the boundary, the tangential derivative
of the right leading symbol is
\begin{align*}
\partial_{x_a}
\bigl(\ii c_{g_2}(\xi)|\xi|_{g_2}^{-2m}\bigr)
={}&
\ii\partial_{x_a}c_{g_2}(\xi)|\xi|_{g_2}^{-2m}
\\
&-
m\ii c_{g_2}(\xi)|\xi|_{g_2}^{-2m-2}
\partial_{x_a}|\xi|_{g_2}^2.
\end{align*}
At $(x_0,0)$, both terms vanish. Indeed,
$\partial_{x_a}g_2^{ij}=0$ gives
$\partial_{x_a}c_{g_2}(\xi)
=-(\partial_{x_a}g_2^{ij})\xi_i\iotaop(\partial_j)=0$,
and $\partial_{x_a}|\xi|_{g_2}^2=0$.
These identities hold for every $\xi_n$, and remain zero
after differentiation in $\xi_n$. Hence $\Phi_1^{(1)}=0$.

\paragraph{Case (a)(II).}

Take $r=1$, $\ell=-n+1$, $j=1$, and $k=|\alpha|=0$.
The coefficient in \eqref{eq:FGLS} is $-1/2$. Equations
\eqref{eq:odd-normal} and \eqref{eq:odd-Q} give
\begin{align*}
\Phi_2^{(1)}
&=
-\frac12
\int_{|\xi'|_{g^{\partial M}}=1}\int_{\mathbb R}
\trE\left[
\partial_{x_n}\pi_{\xi_n}^+\sigma_1(P_1)\,
\partial_{\xi_n}^2\sigma_{-n+1}(Q_1)
\right]
\,\dd\xi_n\,\dd S_{g^{\partial M}}(\xi')
\\
&=
-\frac{\ii(h_1'(0)-h_2'(0))}{4}
\int_{|\xi'|_{g^{\partial M}}=1}\int_{\mathbb R}
\frac1{\xi_n-\ii}
\\
&\hspace{12mm}\cdot
\trE\left[
(c(\xi')+\ii c(dx_n))
\frac{d^2}{d\xi_n^2}
\bigl(c(\xi)(1+\xi_n^2)^{-m}\bigr)
\right]
\,\dd\xi_n\,\dd S_{g^{\partial M}}(\xi').
\end{align*}
The two-factor Clifford traces yield
\begin{align*}
\trE\bigl[(c(\xi')+\ii c(dx_n))c(\xi)\bigr]
&=
\trE[c(\xi')^2]
+\ii\xi_n\trE[c(dx_n)^2]
\\
&\quad
+\xi_n\trE[c(\xi')c(dx_n)]
+\ii\trE[c(dx_n)c(\xi')]
\\
&=-2^n(1+\ii\xi_n).
\end{align*}
The mixed traces vanish because $\xi'$ and $dx_n$ are
orthogonal. Since $c(\xi')+\ii c(dx_n)$ is independent
of $\xi_n$, differentiation can be taken after the trace.
We obtain
\begin{align}
\Phi_2^{(1)}
&=
\ii\,2^{n-2}(h_1'(0)-h_2'(0))
\int_{|\xi'|_{g^{\partial M}}=1}\int_{\mathbb R}
\frac{
\frac{d^2}{d\xi_n^2}
\bigl((1+\ii\xi_n)(1+\xi_n^2)^{-m}\bigr)}
{\xi_n-\ii}
\,\dd\xi_n\,\dd S_{g^{\partial M}}(\xi')
\notag\\
&=
\ii\,2^{n-1}(h_1'(0)-h_2'(0))
\int_{|\xi'|_{g^{\partial M}}=1}\int_{\mathbb R}
\frac{(1+\ii\xi_n)(1+\xi_n^2)^{-m}}
     {(\xi_n-\ii)^3}
\,\dd\xi_n\,\dd S_{g^{\partial M}}(\xi')
\notag\\
&=
-2^{n-1}(h_1'(0)-h_2'(0))\nu_{n-2}L_m
\notag\\
&=
\frac{\pi\nu_{n-2}(h_1'(0)-h_2'(0))}{2}
\frac{(2m)!}{(m-1)!(m+1)!}.
\label{eq:odd-b}
\end{align}
The second equality follows by integrating twice by parts,
using
$\frac{d^2}{d\xi_n^2}(\xi_n-\ii)^{-1}
=2(\xi_n-\ii)^{-3}$.
All boundary terms vanish because
$(1+\ii\xi_n)(1+\xi_n^2)^{-m}=O(|\xi_n|^{1-2m})$
and its derivatives also decay.
The third equality uses
$1+\ii\xi_n=\ii(\xi_n-\ii)$, and the last follows
from \eqref{eq:Lm}.

\paragraph{Case (a)(III).}

Here $r=1$, $\ell=-n+1$, $k=1$, and $j=|\alpha|=0$,
with coefficient $-1/2$. Thus
\[
\Phi_3^{(1)}
=
-\frac12
\int_{|\xi'|_{g^{\partial M}}=1}\int_{\mathbb R}
\trE\left[
\partial_{\xi_n}\pi_{\xi_n}^+\sigma_1(P_1)\,
\partial_{\xi_n}\partial_{x_n}
\sigma_{-n+1}(Q_1)
\right]
\,\dd\xi_n\,\dd S_{g^{\partial M}}(\xi').
\]
Differentiating the right symbol before restricting to
the boundary and the unit tangential sphere gives
\begin{align*}
\partial_{x_n}\sigma_{-n+1}(Q_1)
={}&
-\ii h_2'(0)\sum_{a<n}\xi_a\iotaop(\partial_a)
(1+\xi_n^2)^{-m}
-
m\ii h_2'(0)c(\xi)(1+\xi_n^2)^{-m-1}.
\end{align*}
The left factor, however, vanishes:
$\partial_{\xi_n}\pi_{\xi_n}^+\sigma_1(P_1)=0$
by \eqref{eq:odd-normal}. Therefore $\Phi_3^{(1)}=0$.

\paragraph{Case (b).}

Now $r=0$, $\ell=-n+1$, and $j=k=|\alpha|=0$.
The coefficient in \eqref{eq:FGLS} is $-\ii$, so
\begin{align*}
\Phi_4^{(1)}
&=
-\ii
\int_{|\xi'|_{g^{\partial M}}=1}\int_{\mathbb R}
\trE\left[
\pi_{\xi_n}^+\sigma_0(P_1)\,
\partial_{\xi_n}\sigma_{-n+1}(Q_1)
\right]
\,\dd\xi_n\,\dd S_{g^{\partial M}}(\xi')
\\
&=
\ii
\int_{|\xi'|_{g^{\partial M}}=1}\int_{\mathbb R}
\trE\left[
\partial_{\xi_n}
 \bigl(\pi_{\xi_n}^+\sigma_0(P_1)\bigr)\,
\sigma_{-n+1}(Q_1)
\right]
\,\dd\xi_n\,\dd S_{g^{\partial M}}(\xi').
\end{align*}
The second equality is integration by parts. Its boundary
term vanishes, since the projected left symbol is
$O(|\xi_n|^{-1})$ and the right leading symbol is
$O(|\xi_n|^{1-2m})$.

The numerator in \eqref{eq:odd-zero} does not depend on
$\xi_n$. Consequently,
\[
\partial_{\xi_n}
 \bigl(\pi_{\xi_n}^+\sigma_0(P_1)\bigr)
=
-\frac1{\xi_n-\ii}\,
 \pi_{\xi_n}^+\sigma_0(P_1).
\]
Moreover, \eqref{eq:odd-zero} and
\eqref{eq:odd-contraction} give
\begin{align*}
\trE\bigl[
\pi_{\xi_n}^+\sigma_0(P_1)c(\xi)
\bigr]
&=
-\frac{h_1'(0)-h_2'(0)}{2\ii(\xi_n-\ii)}
\,2^{n-1}(\xi_n-\ii)
\\
&=
\ii\,2^{n-2}(h_1'(0)-h_2'(0)).
\end{align*}
Multiplication by the scalar factor in
\eqref{eq:odd-Q} then yields
\[
\trE\bigl[
\pi_{\xi_n}^+\sigma_0(P_1)\sigma_{-n+1}(Q_1)
\bigr]
=
-2^{n-2}(h_1'(0)-h_2'(0))(1+\xi_n^2)^{-m}.
\]
Substituting into the integrated expression for
$\Phi_4^{(1)}$, we obtain
\begin{align}
\Phi_4^{(1)}
&=
-\ii
\int_{|\xi'|_{g^{\partial M}}=1}\int_{\mathbb R}
\frac{
\trE\bigl[
\pi_{\xi_n}^+\sigma_0(P_1)\sigma_{-n+1}(Q_1)
\bigr]}
{\xi_n-\ii}
\,\dd\xi_n\,\dd S_{g^{\partial M}}(\xi')
\notag\\
&=
\ii\,2^{n-2}(h_1'(0)-h_2'(0))\nu_{n-2}
\int_{\mathbb R}
\frac{(1+\xi_n^2)^{-m}}{\xi_n-\ii}\,\dd\xi_n
\notag\\
&=
\ii\,2^{n-2}(h_1'(0)-h_2'(0))\nu_{n-2}M_m
\notag\\
&=
-\frac{\pi\nu_{n-2}(h_1'(0)-h_2'(0))}{2}
\frac{(2m-1)!}{(m-1)!m!}.
\label{eq:odd-d}
\end{align}
The last equality follows from \eqref{eq:Mm},
$n=2m$, and $\ii^2=-1$.

\paragraph{Case (c).}

Finally, $r=1$, $\ell=-n$, and $j=k=|\alpha|=0$,
with coefficient $-\ii$. We have
\begin{align*}
\Phi_5^{(1)}
&=
-\ii
\int_{|\xi'|_{g^{\partial M}}=1}\int_{\mathbb R}
\trE\left[
\pi_{\xi_n}^+\sigma_1(P_1)\,
\partial_{\xi_n}\sigma_{-n}(Q_1)
\right]
\,\dd\xi_n\,\dd S_{g^{\partial M}}(\xi')
\\
&=0
\end{align*}
by \eqref{eq:odd-normal}.
Thus the subleading right symbol does not contribute.

\begin{proposition}[Odd boundary contribution]
\label{prop:odd-boundary}
Under \eqref{eq:collar},
\[
\int_{\partial M}\Phi_1\,\dd\vol_{g^{\partial M}}
=
\frac{\pi\nu_{n-2}}2
\frac{(2m-1)!}{(m-2)!(m+1)!}
\int_{\partial M}
(h_1'(0)-h_2'(0))\,\dd\vol_{g^{\partial M}}.
\]
\end{proposition}

\begin{proof}
Only $\Phi_2^{(1)}$ and $\Phi_4^{(1)}$ are nonzero.
Using \eqref{eq:odd-b} and \eqref{eq:odd-d},
\begin{align*}
\Phi_1
&=
\Phi_2^{(1)}+\Phi_4^{(1)}
\\
&=
\frac{\pi\nu_{n-2}(h_1'(0)-h_2'(0))}{2}
\left[
\frac{(2m)!}{(m-1)!(m+1)!}
-\frac{(2m-1)!}{(m-1)!m!}
\right]
\\
&=
\frac{\pi\nu_{n-2}(h_1'(0)-h_2'(0))}{2}
\frac{(2m-1)!}{(m-1)!(m+1)!}
\bigl(2m-(m+1)\bigr)
\\
&=
\frac{\pi\nu_{n-2}(h_1'(0)-h_2'(0))}{2}
\frac{(2m-1)!}{(m-2)!(m+1)!}.
\end{align*}
Here $(m-1)/(m-1)!=1/(m-2)!$ since $m\ge2$.
Integrating over $\partial M$ proves the formula.
\end{proof}

\subsection{The bimetric KKW formulas}
\label{sec:boundary-main}

For $n=2m\ge4$, set
$\kappa_n
=\pi\nu_{n-2}(n-2)!/((m-2)!(m+1)!)$.
Combining the boundary computations with
\eqref{eq:boundary-interior} gives the following result.

\begin{theorem}
\label{thm:boundary-KKW}
Let $M^n$ be a compact oriented manifold with smooth boundary,
where $n=2m\ge4$, and let $g_1,g_2$ be smooth Riemannian
metrics satisfying \eqref{eq:collar}.
With the inward normal convention
\eqref{eq:mean-curvature},
\begin{align}
\mathcal R_n^{(2)}(g_1\mid g_2)
&=
2^{n-2}\nu_{n-1}
\int_M\mathscr L_n(g_1\mid g_2)\,\dd\vol_{g_2}
-2\kappa_n
\int_{\partial M}
(K_{g_1}-K_{g_2})\,\dd\vol_{g^{\partial M}},
\label{eq:main-even}
\\
\mathcal R_n^{(1)}(g_1\mid g_2)
&=
2^{n-2}\nu_{n-1}
\int_M\mathscr L_n(g_1\mid g_2)\,\dd\vol_{g_2}
-\kappa_n
\int_{\partial M}
(K_{g_1}-K_{g_2})\,\dd\vol_{g^{\partial M}}.
\label{eq:main-odd}
\end{align}
Here $\mathscr L_n$ is the local scalar
\eqref{eq:local-density}.
\end{theorem}

\begin{proof}
For $s=1,2$, the interior product satisfies
$P_sQ_s=D_{g_1}^2D_{g_2}^{-n}$ modulo smoothing operators.
Thus the interior contribution in
\eqref{eq:FGLS-split} is given by
\eqref{eq:boundary-interior}.

For the boundary contribution,
\eqref{eq:warped-mean-curvature} gives
\[
h_1'(0)-h_2'(0)
=
-\frac{2}{n-1}(K_{g_1}-K_{g_2}).
\]
Since $n=2m$,
\[
\frac{\pi\nu_{n-2}}{n-1}
\frac{(2m-1)!}{(m-2)!(m+1)!}
=
\pi\nu_{n-2}
\frac{(n-2)!}{(m-2)!(m+1)!}
=\kappa_n.
\]
Proposition~\ref{prop:even-boundary} therefore yields
\[
\int_{\partial M}\Phi_2\,\dd\vol_{g^{\partial M}}
=
-2\kappa_n
\int_{\partial M}
(K_{g_1}-K_{g_2})\,\dd\vol_{g^{\partial M}}.
\]
Proposition~\ref{prop:odd-boundary} gives
\[
\int_{\partial M}\Phi_1\,\dd\vol_{g^{\partial M}}
=
-\kappa_n
\int_{\partial M}
(K_{g_1}-K_{g_2})\,\dd\vol_{g^{\partial M}}.
\]
Substituting these expressions into
\eqref{eq:FGLS-split} proves
\eqref{eq:main-even} and \eqref{eq:main-odd}.
\end{proof}

\begin{corollary}
\label{cor:separation}
Under the hypotheses of
Theorem~\ref{thm:boundary-KKW},
\begin{align*}
2\mathcal R_n^{(1)}(g_1\mid g_2)
-\mathcal R_n^{(2)}(g_1\mid g_2)
&=
2^{n-2}\nu_{n-1}
\int_M\mathscr L_n(g_1\mid g_2)\,\dd\vol_{g_2},
\\
\mathcal R_n^{(2)}(g_1\mid g_2)
-\mathcal R_n^{(1)}(g_1\mid g_2)
&=
-\kappa_n
\int_{\partial M}
(K_{g_1}-K_{g_2})\,\dd\vol_{g^{\partial M}}.
\end{align*}
\end{corollary}

\begin{proof}
Subtracting \eqref{eq:main-even} from twice
\eqref{eq:main-odd} cancels the boundary integrals
and gives the first identity. Subtracting
\eqref{eq:main-odd} from \eqref{eq:main-even}
cancels the interior integrals and gives the second.
\end{proof}

\begin{corollary}
\label{cor:first-order-boundary}
Under the hypotheses of
Theorem~\ref{thm:boundary-KKW}, if
$h_1'(0)=h_2'(0)$ on every boundary component, then
the singular Green contribution vanishes for both
splittings. In particular, if $g_1=g_2=g$ on $M$, then
\[
\mathcal R_n^{(2)}(g\mid g)
=
\mathcal R_n^{(1)}(g\mid g)
=
-\frac{(n-2)2^n}{24}\nu_{n-1}
\int_M R_g\,\dd\vol_g.
\]
\end{corollary}

\begin{proof}
The derivative condition implies
$K_{g_1}=K_{g_2}$ by
\eqref{eq:warped-mean-curvature}, so both boundary
terms in Theorem~\ref{thm:boundary-KKW} vanish.
If $g_1=g_2=g$, then $C=0$ and $c_i=0$.
Since $g^{ij}\Ric(g)_{ij}=R_g$ and
$g_{ij}g^{ij}=n$, the local scalar becomes
\[
\mathscr L_n(g\mid g)
=
\left(1-\frac23-\frac n6\right)R_g
=
-\frac{n-2}{6}R_g.
\]
Substitution into either formula of
Theorem~\ref{thm:boundary-KKW} gives the result.
\end{proof}

For an explicit example, suppose that $\partial M$ is
connected. Let $a_1,a_2$ be real constants and take
$g_r=e^{2a_rx_n}g^{\partial M}+dx_n^2$ in the collar.
Then $h_r=e^{-2a_rx_n}$, so
$h_r'(0)=-2a_r$ and $K_{g_r}=(n-1)a_r$.
The two boundary contributions are
\begin{align*}
\int_{\partial M}\Phi_2\,\dd\vol_{g^{\partial M}}
&=
-2\kappa_n(n-1)(a_1-a_2)
\operatorname{vol}_{g^{\partial M}}(\partial M),
\\
\int_{\partial M}\Phi_1\,\dd\vol_{g^{\partial M}}
&=
-\kappa_n(n-1)(a_1-a_2)
\operatorname{vol}_{g^{\partial M}}(\partial M).
\end{align*}
They are nonzero when $a_1\ne a_2$, although the two
metrics induce the same boundary metric $g^{\partial M}$.
The prescribed collar metrics extend to smooth
Riemannian metrics on $M$ without changing these
boundary contributions.

\appendix
\section{Mixed Clifford traces}
\label{app:mixed-Clifford-traces}

In this appendix, we give the mixed two-factor, four-factor, and six-factor traces
and the recursion from which they follow.
Fix $x\in M$, let $r_a\in\{1,2\}$, and let $X_a\in T_xM$.
All traces are ordinary traces on
$E_x=\Lambda^*T_x^*M\otimes\mathbb C$, whose dimension is $2^n$.
The canonical anticommutation relations give
\begin{equation}
\label{eq:mixed-anticommutator}
\begin{aligned}
&\cG{r_a}(X_a)\cG{r_b}(X_b)
+\cG{r_b}(X_b)\cG{r_a}(X_a)
=-\bigl[
g_{r_a}(X_a,X_b)+g_{r_b}(X_a,X_b)
\bigr]\Id_E.
\end{aligned}
\end{equation}

\begin{proposition}[Direct even-factor recursion]
\label{prop:mixed-recursion}
For every integer $q\geq1$,
\begin{align}
&\trE\!\left[
\cG{r_1}(X_1)\cdots\cG{r_{2q}}(X_{2q})
\right]
\notag\\
&\quad=
\sum_{s=2}^{2q}\frac{(-1)^{s+1}}2
\bigl[
g_{r_1}(X_1,X_s)+g_{r_s}(X_1,X_s)
\bigr]
\trE\!\left[
\cG{r_2}(X_2)\cdots
\widehat{\cG{r_s}(X_s)}
\cdots\cG{r_{2q}}(X_{2q})
\right].
\label{eq:mixed-recursion}
\end{align}
The hat denotes omission. An empty product is understood as
$\Id_E$, with $\trE(\Id_E)=2^n$.
Every product of an odd number of Clifford factors has zero trace.
\end{proposition}

\begin{proof}
For this proof, write $\mathfrak c_a=\cG{r_a}(X_a)$.
Equation~\eqref{eq:mixed-anticommutator} gives
\[
\mathfrak c_1\mathfrak c_s
=
-\mathfrak c_s\mathfrak c_1
-\bigl[
g_{r_1}(X_1,X_s)+g_{r_s}(X_1,X_s)
\bigr]\Id_E.
\]
Moving $\mathfrak c_1$ successively through
$\mathfrak c_2,\ldots,\mathfrak c_{2q}$ yields
\begin{align*}
\mathfrak c_1\cdots\mathfrak c_{2q}
={}&
\sum_{s=2}^{2q}(-1)^{s-1}
\bigl[
g_{r_1}(X_1,X_s)+g_{r_s}(X_1,X_s)
\bigr]
\mathfrak c_2\cdots
\widehat{\mathfrak c_s}\cdots
\mathfrak c_{2q}-\mathfrak c_2\cdots\mathfrak c_{2q}\mathfrak c_1.
\end{align*}
Indeed, the contraction with $\mathfrak c_s$ occurs after
$s-2$ interchanges and therefore has sign
$-(-1)^{s-2}=(-1)^{s-1}$.
The last term has sign $(-1)^{2q-1}=-1$.

By cyclicity,
\[
\trE\!\left(
\mathfrak c_2\cdots\mathfrak c_{2q}\mathfrak c_1
\right)
=
\trE\!\left(
\mathfrak c_1\cdots\mathfrak c_{2q}
\right).
\]
Taking the trace of the preceding operator identity, moving the
last trace to the left, and dividing by two proves
\eqref{eq:mixed-recursion}, since
$(-1)^{s-1}=(-1)^{s+1}$.

Finally, let $E_x^{\mathrm{ev}}$ and $E_x^{\mathrm{odd}}$ be the
sums of the even- and odd-degree exterior powers in $E_x$.
Each Clifford factor interchanges these two subspaces.
An odd product is therefore off diagonal with respect to
$E_x=E_x^{\mathrm{ev}}\oplus E_x^{\mathrm{odd}}$
and has zero trace.
\end{proof}

\begin{proposition}[Mixed traces with two, four, and six factors]
\label{prop:mixed-246}
The two-factor trace is
\begin{equation}
\label{eq:mixed-two}
\begin{aligned}
&\trE\!\left[
\cG{r_1}(X_1)\cG{r_2}(X_2)
\right]
=-\frac{\trE(\Id_E)}2
\bigl[
g_{r_1}(X_1,X_2)+g_{r_2}(X_1,X_2)
\bigr].
\end{aligned}
\end{equation}
The four-factor trace is
\begin{align}
&\trE\!\left[
\cG{r_1}(X_1)\cG{r_2}(X_2)
\cG{r_3}(X_3)\cG{r_4}(X_4)
\right]
\notag\\
&=\frac{\trE(\Id_E)}4\Bigl\{
\notag\\
&\quad
\bigl[g_{r_1}(X_1,X_2)+g_{r_2}(X_1,X_2)\bigr]
\bigl[g_{r_3}(X_3,X_4)+g_{r_4}(X_3,X_4)\bigr]
\notag\\
&\quad-
\bigl[g_{r_1}(X_1,X_3)+g_{r_3}(X_1,X_3)\bigr]
\bigl[g_{r_2}(X_2,X_4)+g_{r_4}(X_2,X_4)\bigr]
\notag\\
&\quad+
\bigl[g_{r_1}(X_1,X_4)+g_{r_4}(X_1,X_4)\bigr]
\bigl[g_{r_2}(X_2,X_3)+g_{r_3}(X_2,X_3)\bigr]
\Bigr\}.
\label{eq:mixed-four}
\end{align}
The six-factor trace is
\begingroup
\allowdisplaybreaks[1]
\begin{align}
&\trE\!\left[
\cG{r_1}(X_1)\cG{r_2}(X_2)\cG{r_3}(X_3)
\cG{r_4}(X_4)\cG{r_5}(X_5)\cG{r_6}(X_6)
\right]
\notag\\*
&=\frac{\trE(\Id_E)}8\Bigl\{
\notag\\*
&\quad-
\bigl[g_{r_1}(X_1,X_2)+g_{r_2}(X_1,X_2)\bigr]
\bigl[g_{r_3}(X_3,X_4)+g_{r_4}(X_3,X_4)\bigr]
\bigl[g_{r_5}(X_5,X_6)+g_{r_6}(X_5,X_6)\bigr]
\notag\\
&\quad+
\bigl[g_{r_1}(X_1,X_2)+g_{r_2}(X_1,X_2)\bigr]
\bigl[g_{r_3}(X_3,X_5)+g_{r_5}(X_3,X_5)\bigr]
\bigl[g_{r_4}(X_4,X_6)+g_{r_6}(X_4,X_6)\bigr]
\notag\\
&\quad-
\bigl[g_{r_1}(X_1,X_2)+g_{r_2}(X_1,X_2)\bigr]
\bigl[g_{r_3}(X_3,X_6)+g_{r_6}(X_3,X_6)\bigr]
\bigl[g_{r_4}(X_4,X_5)+g_{r_5}(X_4,X_5)\bigr]
\notag\\
&\quad+
\bigl[g_{r_1}(X_1,X_3)+g_{r_3}(X_1,X_3)\bigr]
\bigl[g_{r_2}(X_2,X_4)+g_{r_4}(X_2,X_4)\bigr]
\bigl[g_{r_5}(X_5,X_6)+g_{r_6}(X_5,X_6)\bigr]
\notag\\
&\quad-
\bigl[g_{r_1}(X_1,X_3)+g_{r_3}(X_1,X_3)\bigr]
\bigl[g_{r_2}(X_2,X_5)+g_{r_5}(X_2,X_5)\bigr]
\bigl[g_{r_4}(X_4,X_6)+g_{r_6}(X_4,X_6)\bigr]
\notag\\
&\quad+
\bigl[g_{r_1}(X_1,X_3)+g_{r_3}(X_1,X_3)\bigr]
\bigl[g_{r_2}(X_2,X_6)+g_{r_6}(X_2,X_6)\bigr]
\bigl[g_{r_4}(X_4,X_5)+g_{r_5}(X_4,X_5)\bigr]
\notag\\
&\quad-
\bigl[g_{r_1}(X_1,X_4)+g_{r_4}(X_1,X_4)\bigr]
\bigl[g_{r_2}(X_2,X_3)+g_{r_3}(X_2,X_3)\bigr]
\bigl[g_{r_5}(X_5,X_6)+g_{r_6}(X_5,X_6)\bigr]
\notag\\
&\quad+
\bigl[g_{r_1}(X_1,X_4)+g_{r_4}(X_1,X_4)\bigr]
\bigl[g_{r_2}(X_2,X_5)+g_{r_5}(X_2,X_5)\bigr]
\bigl[g_{r_3}(X_3,X_6)+g_{r_6}(X_3,X_6)\bigr]
\notag\\
&\quad-
\bigl[g_{r_1}(X_1,X_4)+g_{r_4}(X_1,X_4)\bigr]
\bigl[g_{r_2}(X_2,X_6)+g_{r_6}(X_2,X_6)\bigr]
\bigl[g_{r_3}(X_3,X_5)+g_{r_5}(X_3,X_5)\bigr]
\notag\\
&\quad+
\bigl[g_{r_1}(X_1,X_5)+g_{r_5}(X_1,X_5)\bigr]
\bigl[g_{r_2}(X_2,X_3)+g_{r_3}(X_2,X_3)\bigr]
\bigl[g_{r_4}(X_4,X_6)+g_{r_6}(X_4,X_6)\bigr]
\notag\\
&\quad-
\bigl[g_{r_1}(X_1,X_5)+g_{r_5}(X_1,X_5)\bigr]
\bigl[g_{r_2}(X_2,X_4)+g_{r_4}(X_2,X_4)\bigr]
\bigl[g_{r_3}(X_3,X_6)+g_{r_6}(X_3,X_6)\bigr]
\notag\\
&\quad+
\bigl[g_{r_1}(X_1,X_5)+g_{r_5}(X_1,X_5)\bigr]
\bigl[g_{r_2}(X_2,X_6)+g_{r_6}(X_2,X_6)\bigr]
\bigl[g_{r_3}(X_3,X_4)+g_{r_4}(X_3,X_4)\bigr]
\notag\\
&\quad-
\bigl[g_{r_1}(X_1,X_6)+g_{r_6}(X_1,X_6)\bigr]
\bigl[g_{r_2}(X_2,X_3)+g_{r_3}(X_2,X_3)\bigr]
\bigl[g_{r_4}(X_4,X_5)+g_{r_5}(X_4,X_5)\bigr]
\notag\\
&\quad+
\bigl[g_{r_1}(X_1,X_6)+g_{r_6}(X_1,X_6)\bigr]
\bigl[g_{r_2}(X_2,X_4)+g_{r_4}(X_2,X_4)\bigr]
\bigl[g_{r_3}(X_3,X_5)+g_{r_5}(X_3,X_5)\bigr]
\notag\\
&\quad-
\bigl[g_{r_1}(X_1,X_6)+g_{r_6}(X_1,X_6)\bigr]
\bigl[g_{r_2}(X_2,X_5)+g_{r_5}(X_2,X_5)\bigr]
\bigl[g_{r_3}(X_3,X_4)+g_{r_4}(X_3,X_4)\bigr]
\Bigr\}.
\label{eq:mixed-six}
\end{align}
\endgroup
\end{proposition}

\begin{proof}
For two factors, taking the trace of
\eqref{eq:mixed-anticommutator} and using cyclicity gives
\[
2\trE\!\left[
\cG{r_1}(X_1)\cG{r_2}(X_2)
\right]
=
-\bigl[
g_{r_1}(X_1,X_2)+g_{r_2}(X_1,X_2)
\bigr]\trE(\Id_E),
\]
which proves \eqref{eq:mixed-two}.

For four factors, apply
Proposition~\ref{prop:mixed-recursion} with $q=2$.
Pairing the first factor with the second, third, and fourth
gives coefficients $-\frac12$, $\frac12$, and $-\frac12$,
respectively. Each remaining two-factor trace contributes the
factor $-\frac12\trE(\Id_E)$ by \eqref{eq:mixed-two}.
Thus the three coefficients are
$\frac14\trE(\Id_E)$,
$-\frac14\trE(\Id_E)$, and
$\frac14\trE(\Id_E)$.
Their respective pairings are
$\{(1,2),(3,4)\}$,
$\{(1,3),(2,4)\}$, and
$\{(1,4),(2,3)\}$.
This gives \eqref{eq:mixed-four}.

For six factors, the recursion with $q=3$ pairs the first factor
successively with the factors $2,3,4,5,6$.
The corresponding coefficients are
$-\frac12$, $\frac12$, $-\frac12$, $\frac12$, and $-\frac12$.
For each choice, \eqref{eq:mixed-four} supplies three terms with
signs $+,-,+$ and common factor $\trE(\Id_E)/4$.
Consequently, the five groups have sign patterns
\[
(-,+,-),\qquad
(+,-,+),\qquad
(-,+,-),\qquad
(+,-,+),\qquad
(-,+,-),
\]
with common factor $\trE(\Id_E)/8$.
Each pairing has a unique partner for the first index, and the
remaining four indices admit precisely three pairings.
These five groups therefore contain all fifteen terms, without
repetition, and give \eqref{eq:mixed-six}.
\end{proof}

\section*{Acknowledgements}

This work was supported by Science and Technology Development Plan Project of Jilin Province,
China: No.20260102245JC, NSFC (Grant Nos. 12301063 and 11771070) and 2024 Liaoning Provincial Natural Science Foundation Program (Ph.D. Research Start-up Project) (Grant No. 2024-BS-205).
\begingroup
\sloppy

\endgroup

\end{document}